\documentclass[a4paper,reqno]{amsart}

\usepackage{amsmath, amssymb, amsthm}
\usepackage[english]{babel}
\usepackage{enumerate}
\usepackage[all]{xy}
\usepackage{graphicx}
\usepackage{mathtools}
\usepackage{mathrsfs}
\usepackage{hyperref}
\usepackage{verbatim}
\usepackage{color}

\newcommand{\rosso}[1]{{\color{red}{#1}}}

\allowdisplaybreaks
\DeclareMathOperator{\Hom}{Hom}
\newcommand*{\id}{\textup{id}}%identity map
\numberwithin{equation}{section}

\theoremstyle{plain}

\newtheorem{thm}{Theorem}[section]
\newtheorem{lem}[thm]{Lemma}
\newtheorem{prop}[thm]{Proposition}
 \newtheorem{cor}[thm]{Corollary}
\newtheorem{defi}[thm]{Definition}

\theoremstyle{remark}
\newtheorem{exa}[thm]{Example}
\newtheorem{rem}[thm]{Remark}

\numberwithin{equation}{section}

\newcommand{\ot}{\otimes}

\newcommand{\beq}{\begin{equation}}
\newcommand{\eeq}{\end{equation}}

\newcommand{\cL}{\mathcal{L}}
\newcommand{\CB}{\mathcal{B}}

\newcommand{\C}{\mathcal{C}}

\newcommand{\zero}[1]{{#1}{}_{\scriptscriptstyle{(0)}}}
\newcommand{\one}[1]{{#1}{}_{\scriptscriptstyle{(1)}}}
\newcommand{\two}[1]{{#1}{}_{\scriptscriptstyle{(2)}}}
\newcommand{\three}[1]{{#1}{}_{\scriptscriptstyle{(3)}}}
\newcommand{\four}[1]{{#1}{}_{\scriptscriptstyle{(4)}}}
\newcommand{\five}[1]{{#1}{}_{\scriptscriptstyle{(5)}}}
\newcommand{\six}[1]{{#1}{}_{\scriptscriptstyle{(6)}}}
\newcommand{\seven}[1]{{#1}{}_{\scriptscriptstyle{(7)}}}
\newcommand{\eight}[1]{{#1}{}_{\scriptscriptstyle{(8)}}}
\newcommand{\nine}[1]{{#1}{}_{\scriptscriptstyle{(9)}}}
\newcommand{\ten}[1]{{#1}{}_{\scriptscriptstyle{(10)}}}
\newcommand{\eleven}[1]{{#1}{}_{\scriptscriptstyle{(11)}}}
\newcommand{\twelve}[1]{{#1}{}_{\scriptscriptstyle{(12)}}}
\newcommand{\thirteen}[1]{{#1}{}_{\scriptscriptstyle{(13)}}}
\newcommand{\fourteen}[1]{{#1}{}_{\scriptscriptstyle{(14)}}}
\newcommand{\fifteen}[1]{{#1}{}_{\scriptscriptstyle{(15)}}}

\newcommand{\tuno}[1]{{#1}{}{}^{\scriptscriptstyle{<1>}}}
\newcommand{\tdue}[1]{{#1}{}{}^{\scriptscriptstyle{<2>}}}

\newcommand{\bicross}{{\blacktriangleright\!\!\!\triangleleft}}

\newcommand{\lbiprod}{{>\!\!\!\triangleleft\kern-.33em\cdot}}
\newcommand{\rbiprod}{{\cdot\kern-.33em\triangleright\!\!\!<}}
\newcommand{\rcocross}{{\blacktriangleright\!\!<}}

\renewcommand{\o}{{}_{\scriptscriptstyle{(1)}}}
\renewcommand{\t}{{}_{\scriptscriptstyle{(2)}}}
\renewcommand{\th}{{}_{\scriptscriptstyle{(3)}}}
\newcommand{\fo}{{}_{\scriptscriptstyle{(4)}}}
\newcommand{\uo}{{}^{\scriptscriptstyle{(1)}}}
\newcommand{\ut}{{}^{\scriptscriptstyle{(2)}}}

\newcommand{\la}{{\triangleright}}

\newcommand{\eps}{{\epsilon}}
\DeclareMathOperator{\tens}{\otimes}
\newcommand{\CR}{\mathcal{R}}
\newcommand{\CH}{\mathcal{H}}
\newcommand{\CM}{\mathcal{M}}
\newcommand{\CL}{\mathcal{L}}

\newcommand{\Ss}{{S}}
\newcommand{\can}{{\rm can}}

\begin{document}

\author{Xiao Han and Shahn Majid}
\address{Queen Mary University of London\\
		School of Mathematical Sciences, Mile End Rd, London E1 4NS, UK}

\email{x.h.han@qmul.ac.uk, s.majid@qmul.ac.uk}
%\thanks{\blu{XH supported by Marie Curie Fellowship grant number ***}}


\title{Hopf-Galois Extensions and Twisted Hopf algebroids}
%\date{29 February 2020}
%

\begin{abstract} We show that the Ehresmann-Schauenburg bialgebroid of a quantum principal bundle $P$ or Hopf Galois extension with structure quantum group $H$ is in fact a left Hopf algebroid  $\CL(P,H)$.  We show further that if $H$ is coquasitriangular then $\CL(P,H)$ has an antipode map $S$ obeying certain minimal axioms. Trivial quantum principal bundles or cleft Hopf Galois extensions with base $B$ are known to be cocycle cross products $B\#_\sigma H$ for a cocycle-action pair $(\la,\sigma)$ and we look at these of a certain `associative type' where $\la$ is an actual action. In this case also, we show that
the associated left Hopf algebroid has an antipode obeying our minimal axioms. We show that if $\CL$ is any left Hopf algebroid then so is its cotwist $\CL^\varsigma$ as an extension of the previous bialgebroid Drinfeld cotwist theory. We show that in the case of associative type, $\CL(B\#_\sigma H,H)=\CL(B\# H)^{\tilde\sigma}$ for a Hopf algebroid cotwist $\varsigma=\tilde\sigma$. Thus, switching on $\sigma$ of associative type appears at the Hopf algebroid level as a Drinfeld cotwist.  We view the affine quantum group $\widehat{U_q(sl_2)}$ and the quantum Weyl group of $u_q(sl_2)$ as examples of associative type. \end{abstract}

\maketitle
%\tableofcontents
%\parskip = .75 ex

\section{Introduction}

A right Hopf-Galois extension or quantum principal bundle is a right comodule algebra $P$ under a Hopf algebra $H$ such that the induced map $P\tens_B P\to P\tens H$ is an isomorphism of vector spaces, where $B=P^{co(H)}$ is the fixed subalgebra. In noncommutative geometry, one thinks of $P$ as like the algebra of functions on the total space of a principal bundle, $H$ as like the algebra of functions on the fibre group and $B$ as like the functions on the base, except that all algebras may be noncommutative and the group structure of the fibre is now a Hopf algebra or quantum group. This is the starting point for quantum group gauge theory\cite{BrzMa} and we refer to \cite[Chap.~5]{BegMa} for a recent account of the geometric picture, differential structures and connections. In the classical case, one has an associated Ehresmann groupoid of a principal bundle and similarly for any Hopf Galois extension, one has\cite{schau} the `Ehresmann-Schauenburg bialgebroid' which we will show is in fact a left Hopf algebroid and denote $\CL(P, H)$. This is a main result of the paper and appears in Theorem~\ref{ESHopf}.

In general, while the axioms of left bialgebroids \cite{BW,BS}  are somewhat settled at this point (with a similar notion for right bialgebroids) the correct notion of Hopf algebroid is less clear. By left Hopf algebroid we will mean in the sense of \cite{schau1} that a certain `translation map' is invertible. This reduces to existence of a usual Hopf algebra antipode $S$ when the base algebra $B=k$, but in general an antipode map appears to be a stronger notion and there are several candidates. One of these is in \cite{BS}, where one needs both a left and right bialgebroid with $S$ connecting them and $S$ antimultiplicative. This notion is, however, too restrictive for our purposes and we introduce in  Lemma~\ref{lemS}  what we consider a minimal formulation for an invertible antipode $S$, which drops $S$ being antimultiplicative but is strong enough to imply a left Hopf algebroid. We expect but do not show that it is a strictly stronger concept than a left Hopf algebroid and we show that $\CL(P,H)$ has such an antipode at least if $H$ is coquasitriangular.  Another general result of the paper is Theorem~\ref{thm. 2 cocycle twist} about  Drinfeld (co)twists of left Hopf algebroids, extending the bialgebroid version in \cite{Boehm}. We do not find conditions on the cotwists for an antipode to cotwist, but the example in Corollary~\ref{twistS} comes close.

Before these general results, Section~\ref{sec2} starts with the simpler case of a cleft Hopf Galois extension\cite{mont} or `trivial' quantum principal bundle\cite{BrzMa} defined as a Hopf-Galois extension equipped with a cleaving or `trivialisation' map $j:H\to P$ of comodule algebras (where $H$ coacts on itself by the coproduct) which is convolution-invertible as a map from the coalgebra of $H$ to the algebra $P$. One of the first observations to be made here is that, unlike classical geometry, there is a potentially nontrivial nonAbelian cohomology associated to such seemingly trivial extensions, denoted $\CH^2(H,B)$ in \cite{Ma:non}\cite{Ma:book}. Here it is known that a cleft extension allows for an identification
\[   P\cong B\#_\sigma H\]
as comodule algebras for certain `cocycle data' $(\la,\sigma)$ where $\la:H\tens B\to B$  is a measuring (but not necessarily an action) and $\sigma:H\tens H\to B$ obeying certain conditions. Gauge transformation by convolution invertible maps $u:H\to B$ changes $(\la,\sigma)$ and the nonAbelian cohomology can be defined as cocycles modulo these. In the special case where $H$ is cocommutative,  it was shown in \cite{D89} that there is a bijective correspondence between the equivalence classes of $H$-cleft Hopf Galois extensions and the usual cohomology group $\mathcal{H}^{2}(H, Z(B))$, where $Z(B)$ is the centre of $B$ and $\sigma$ is the linear extension of a group 2-cocycle in the case where $H$ is a group algebra. Another special case is where $B=k$ is the ground field and in this case $\sigma$ is a Drinfeld 2-cocycle as in the theory of Hopf algebra cotwists\cite[Chap 2]{Ma:book}. We introduce and study a joint  `associative type' generalisation of these two special cases where $\la$ is in fact an action so that we also have an ordinary smash product $B\#H$.

Cleft extensions are then studied further in Section~\ref{sec3} to illustrate the Hopf algebroid theory more explicitly. First we give a version $B^e\#_\sigma H$ of the Ehresmann-Schauenburg bialgebroid in the cleft cocycle extension case directly on the vector of $B^e\tens H$ where $B^e=B\tens B^{op}$. From our general results,  this is necessarily a left Hopf algebroid but we show that it has an antipode at least when $\sigma$ is of associative type. We also show in the associative type case that there is a Hopf algebroid cotwist $\tilde\sigma$ on the Hopf algebroid $B^e\#H$ such that  $B^e\#_\sigma H=(B^e\#H)^{\tilde\sigma}$ is a Hopf algebroid twist. Thus, the introduction of $\sigma$ at the cocycle level corresponds to a Drinfeld cotwist at the Hopf algebroid level. We also explore how our constructions change both under a Drinfeld cotwist of $\chi$ on $H$ and under a gauge transformation $u$ from our two points of view.

Section~\ref{sec4} contains our results for general quantum principal bundles and the paper then concludes with some examples in Section~\ref{sec5}. Of course, many Hopf-Galois extensions are known but the new notion of associative type is less clear and we show how they arise from certain cocycle bicrossproducts of interest in the theory of quantum groups.

\section{Cocycle cleft extensions and associative type}\label{sec2}

This section starts with recalling basic algebraic preliminaries, then studies cocycle extensions of associative type. We work over a field $k$.

\subsection{Basic algebraic preliminaries} We recall that a Hopf algebra $H$ means a unital algebra which is also a counital coalgebra with coproduct $\Delta$ and counit $\epsilon$, such that these are algebra maps  and there is an antipode $S$ obeying $S(h\o)h\t=1\eps(h)=h\o S(h\t)$ for all $h\in H$. We use here the sumless Sweedler notation $\Delta(h)=\one{h}\otimes \two{h}$ for all $h\in H$. We refer to texts such as \cite{Swe,Ma:book} for more details.  If  $C$ is a coalgebra and $P$ is an algebra, the space $\Hom_k(C,P)$ is unital associative algebra with the `convolution' product  $\phi\star\psi(c):=\phi(\one{c})\psi(\two{c})$, for all $\phi, \psi \Hom_k(C,P)$ and from this point of view, $S$ is the convolution-inverse of the identity map $\id:H\to H$.

Given a Hopf algebra $H$, a \textit{left $H$-module algebra} is an algebra $B$ with a left action $\la$ of $H$ such that $    h\triangleright(ab)=(\one{h}\triangleright a)(\two{h}\triangleright b)$ for all $a, b\in B$ and $h\in H$, and $h\triangleright 1=\epsilon(h)1$. In this case, one has a smash or cross product algebra $B\#H$ built on $B\tens H$ with product $(a\#h)(a'\#g)=a(\one{h}\triangleright a')\#\two{h}g$ for all $a, a'\in B$ and $g, h\in H$. Similarly, a  \textit{right $H$-comodule algebra} is an algebra $P$ with a right coaction $\Delta_R: P\to P\otimes H$ (the axioms are dual to those of a module) such that $\Delta_R$ is an algebra map, where $P\tens H$ has the tensor product algebra structure. We use the sumless Sweedler notation $\Delta_R(p)=\zero{p}\otimes \one{p}$.

The category of left modules ${}_H\CM$ and the category of right comodules $\CM^H$ are both monoidal with finite-dimensional objects rigid. For example, if $V$ and $W$ are right comodules the $V\otimes W$ is  right $H$-comodule     $\Delta_{R, V\tens W}(v\otimes w):=\zero{v}\otimes \zero{w}\otimes \one{v}\one{w}$ for all $v\in V$ and $w\in W$. An $H$-module algebra $B$ just means an algebra in ${}_H\CM$ and a comodule algebra $P$ just means an algebra in $\CM^H$.

Finally, if $H$ is a Hopf algebra and $\chi:H\tens H\to k$ a Drinfeld 2-cocycle or {\em cotwist} in the sense of convolution-invertible and obeying\cite{Ma:book}
\[  \chi(g\o,  f\o)\chi(h, g\t f\t)=\chi(h\o, g\o)\chi(h\o g\o, f),\quad \chi(1, h)=\chi(h, 1)=\eps(h) \]
then one has a new Hopf algebra $H^\chi$ on the same vector space as $H$ but with a modified product
\[  h\bullet_\chi g= \chi(h\o, g\o)h\t g\t\chi^{-1}(h\th, g\th)\]
such that the category $\CM^{H^\chi}$ is monoidally equivalent to $\CM^H$\cite{Ma:book}. Moreover, if $P$ is an $H$-comodule algebra then there is a 1-sided cotwisted comodule algebra $P_\chi$ with new product\cite[Sec.~2.3]{Ma:book}
\begin{align}\label{eqn. P cotwist}
    p\cdot_{\chi}q:=\zero{p}\zero{q}\chi^{-1}(\one{p}, \one{q}),
\end{align}
for all $p, q\in P$ and the same coaction on the underlying vector space.

\subsection{Cocycle cross product algebras revisited}\label{seccocy}

More generally, $\la:H\tens B\to B$ is a measuring of a Hopf algebra $H$ on an algebra $B$ if $h\triangleright(ab)=(\one{h}\triangleright a)(\two{h}\triangleright b)$ for all $a, b\in B$ and $h\in H$, and $h\triangleright 1=\epsilon(h)1$, i.e. like a module algebra but without requiring $\la$ to be an action. Given such a measuring $\la$, we assume a `2-cocycle' in the sense of a convolution-invertible map $\sigma:H\tens H\to B$ such that
\begin{itemize}
    \item [(1)]$1\triangleright b=b$,
    \item [(2)]$h\triangleright(g\triangleright b)=\sigma(\one{h}, \one{g})(\two{h}\two{g}\triangleright b)\sigma^{-1}(\three{h}, \three{g})$,
    \item [(3)] $\sigma(h, 1)=\sigma(1, h)=\epsilon(h)1$,
    \item [(4)] $(\one{h}\triangleright\sigma(\one{g},\one{f}))\sigma(\two{h}, \two{g}\two{f})=\sigma(\one{h}, \one{g})\sigma(\two{h}\two{g}, f)$,
\end{itemize}
for all $h, g, f\in H$ and $b\in B$. Then \cite{DT86} showed that there is an {\em cocycle cross product} algebra structure on $B\#_{\sigma} H$ built on $B\otimes H$ as vector space, with the product
\begin{align}\label{equation. product of twisted crossed product}
    (b\#_{\sigma} h)(c\#_{\sigma} g)=b(\one{h}\triangleright c)\sigma(\two{h}, \one{g})\#_{\sigma} \three{h}\two{g},
\end{align}
for all $h, g\in H$ and $b, c\in B$.  The converse is also true: given that $H$ measures $B$ and $\sigma: H\otimes H\to B$ is a convolution invertible linear map, these conditions are needed for the formula given to define an associative product with unit $1\#_{\sigma}1$. We have the following useful lemma:

\begin{lem}\label{lemma. unital twist} cf. \cite[Prop. 7.2.7]{mont} Given the data for a cocycle cross product $B\#_{\sigma}H$,  we have the following properties of $\sigma$ for all $h,g,f\in H$:
\begin{itemize}
    \item [(1)] $\sigma^{-1}(\one{h}, \one{g}\one{f})(\two{h}\triangleright \sigma^{-1}(\two{g}, \two{f}))=\sigma^{-1}(\one{h}\one{g}, f)\sigma^{-1}(\two{h}, \two{g})$,
    \item[(2)] $h\triangleright \sigma(g, f)=\sigma(\one{h}, \one{g})\sigma(\two{h}\two{g}, \one{f})\sigma^{-1}(\three{h}, \three{g}\two{f})$,
    \item[(3)] $h\triangleright \sigma^{-1}(g, f)=\sigma(\one{h}, \one{g}\one{f})\sigma^{-1}(\two{h}\two{g}, \two{f})\sigma^{-1}(\three{h}, \three{g})$,
    \item[(4)] $(\one{h}\triangleright \sigma^{-1}(S(\four{h}), \five{h}))\sigma(\two{h}, S(\three{h}))=\epsilon(h)1$.
    \item[(5)]$\sigma^{-1}(S(\one{h})S(\one{g}), \two{h}\two{g})=\\
    \sigma^{-1}(S(\three{h}), \four{h})$ $(S(\two{h})\triangleright \sigma^{-1}(S(\two{g}), \three{g}\five{h}))$ $\sigma(S(\one{h}), S(\one{g}))$.
    \item[(6)] $\sigma^{-1}(S(\two{h}), \three{h})S(\one{h})\la\sigma(\four{h}, S(\five{h}))=\eps(h)1$.
\end{itemize}
\end{lem}
\begin{proof} (1) and (2) can be derived by applying the convolution inverse on both side of the 2-cocycle conditions, and (3) can be derived from (1). For (4),
\begin{align*}
    (\one{h}\triangleright &\sigma^{-1}(S(\four{h}), \five{h}))\sigma(\two{h}, \three{h})\\
    &=\sigma(\one{h}, S(\eight{h})\nine{h})\sigma^{-1}(\two{h}S(\seven{h}), \ten{h})\sigma^{-1}(\three{h}, S(\six{h}))\sigma(\four{h}, S(\five{h}))\\
    &=\sigma(\one{h}, S(\four{h})\five{h})\sigma^{-1}(\two{h}S(\three{h}), \six{h})\\
    &=\epsilon(h)1,
\end{align*}
for any $h\in H$, and similarly for (6). For (5)
\begin{align*}
    \sigma^{-1}(S(\three{h}), & \four{h})(S(\two{h})\triangleright \sigma^{-1}(S(\two{g}), \three{g}\five{h}))\sigma(S(\one{h}), S(\one{g}))\\
    =&\sigma^{-1}(S(\five{h}), \six{h})\sigma(S(\four{h}), S(\four{g})\five{g}\seven{h})\sigma^{-1}(S(\three{h})S(\three{g}),\six{g}\eight{h})\\
    &\sigma^{-1}(S(\two{h}),S(\two{g}))\sigma(S(\one{h}), S(\one{g}))\\
    =&\sigma^{-1}(S(\one{h})S(\one{g}), \two{h}\two{g}),
\end{align*}
where the first step use (3).
\end{proof}

From an extension theory point of view, a cleft extension of an algebra $B$ by a Hopf algebra $H$ is a comodule algebra $P$, a convolution-invertible  unital comodule map $j:H\to P$ and an algebra inclusion,
\[ B\hookrightarrow  P \leftarrow H,\]
such that the tensor product of these maps combined with the product of $P$ gives an isomorphism $B\tens H\to P$  of right $H$-comodules and left $B$-modules. We use the canonical left $B$-module and right $H$-comodule structure on $B\tens H$ and the assertion entails that $j$ is injective. It is easy to see that  $B\#_\sigma H$ meets these conditions and conversely, given the extension data, one can show that $B=P^{co(H)}$ is the invariant subalgebra and construct $(\la,\sigma)$ by\cite{Ma:non}\cite[Prop.~6.3.6]{Ma:book}
\begin{equation}\label{equ. cleaving map to 2-cocycle} h\la b=j(h\o)b j^{-1}(h\t),\quad \sigma(h\tens g)=j(h\o)j(g\o)j^{-1}(h\t g\t) \end{equation}
with images in $B$ due to the assumed equivariance of $j$. A map between extensions is an $H$-comodule map which is the identity on $B$. Moreover, if $(\la,\sigma)$ is a cocycle then
\begin{align}\label{equ. equivalent corssed product condition 1}
    h\triangleright^u b=u^{-1}(\one{h})(\two{h}\triangleright b)u(\three{h}),
\end{align}
\begin{align}\label{equ. equivalent corssed product condition 2}
    \sigma^u(h, g)=u^{-1}(\one{h})(\two{h}\triangleright u^{-1}(\one{g}))\sigma(\three{h}, \two{g})u(\four{h}\three{g}).
\end{align}
for a convolution-invertible unital map $u: H\to B$ is another, and  $(\la^u,\sigma^u)$ gives an isomorphic cocycle extension by
\[ \Theta_u:B\#_{\sigma^u}H\to B\#_{\sigma}H,\quad \Theta(b\#_{\sigma^u} h)=bu^{-1}(h\o)\#_\sigma h\t.\]
Conversely, isomorphic extensions are related in this way\cite{Ma:non}\cite[Props.~6.3.4,~6.3.5]{Ma:book}. This therefore defines a kind of `nonAbelian cohomology'  theory $\CH^2(H,B)$ realised as cocycles modulo such `gauge equivalence'. We will recall in Section~\ref{sec4} how such cleft extensions are Hopf-Galois and hence can be viewed as trivial quantum principal bundles.

Cleft cocycle extensions and their duals have also been applied in a theory of cocycle bicrossproduct extensions\cite{Ma:mor}\cite{Ma:book} where $B$ and $P=B{}_\sigma\bicross{}^\psi H$ are Hopf algebras and the extension is both cleft and  `cocleft'. This will be a source of examples in Section~\ref{sec5}.  Finally, a different operation one can perform on cocycle cross products is a Drinfeld cotwist.

\begin{prop}\label{prop. twist cleft Galois object}
If $(\la,\sigma)$ is a cocycle  for $B\#_\sigma H$ then $\la_\chi=\la$ and $\sigma_\chi=\sigma*\chi^{-1}$ are cocycle data for $B\#_{\sigma_\chi}H^\chi$. Moreover, $\CH^2(H^\chi, B)\cong \CH^2(H,B)$.
\end{prop}
\begin{proof} One can simply check that the twisted data meets the conditions for a cocycle. Or, from the point of view of cleft extensions, assume $j: H\to P$ is the cleaving map of the cleft extension. We define $j_{\chi}=j$ on the same underlying vector space of $H$  now viewed as a $H^\chi$-comodule map. We first check that its inverse is given by
\begin{align*}     j_{\chi}^{-1}(h):=j^{-1}(\three{h})\chi(\one{h}, S(\two{h})),
\end{align*}
where $j^{-1}$ is the inverse of the original cleaving map. On the one hand,
\begin{align*}
    j_{\chi}^{-1}\ast j_{\chi}(h)=&\chi(\one{h}, S(\two{h}))j^{-1}(\three{h})\cdot_{\chi}j(\four{h})\\
    =&\chi(\one{h}, S(\two{h}))j^{-1}(\four{h})j(\five{h})\chi^{-1}(S(\three{h}), \six{h})\\
    =&\chi(\one{h}, S(\two{h}))\chi^{-1}(S(\three{h}), \four{h})=\eps(h),
\end{align*}
where the second step uses the fact that $\zero{j^{-1}(h)}\otimes \one{j^{-1}(h)}=j^{-1}(\two{h})\otimes S(\one{h})$ for any $h\in H$, the last step uses the fact that $\chi$ is a 2-cocycle on the Hopf algebra $H$. On the other hand,
\begin{align*}
    j_{\chi}\ast j_{\chi}^{-1} (h)=&j(\one{h})\cdot_{\chi}j^{-1}(\four{h})\chi(\two{h}, S(\three{h}))\\
    =&j(\one{h})j^{-1}(\six{h})\chi(\two{h}, S(\five{h}))\chi(\three{h}, S(\four{h}))=\eps(h),
\end{align*}
According to (\ref{equ. cleaving map to 2-cocycle}), the deformed 2-cocycle action $\triangleright_{\chi}$ is given by
\begin{align*}
    h\triangleright_{\chi} b=&j_{\chi}(\one{h})\cdot_{\chi}b\cdot_{\chi}j^{-1}_{\chi}(\two{h})\\
    =&(1\#_{\sigma}\one{h})\cdot_{\chi}(b\chi(\two{h}, S(\three{h}))j^{-1}(\four{h}))\\
    =&(1\#_{\sigma}\one{h})\cdot_{\chi}(b\chi(\two{h}, S(\three{h}))\sigma^{-1}(S(\five{h}), \six{h})\#_{\sigma}S(\four{h}))\\
    =&(1\#_{\sigma}\one{h})(b\chi(\three{h}, S(\four{h}))\sigma^{-1}(S(\seven{h}), \eight{h})\#_{\sigma}S(\six{h}))\chi^{-1}(\two{h}, S(\five{h}))\\
    =&(1\#_{\sigma}\one{h})(b\sigma^{-1}(S(\three{h}), \four{h})\#_{\sigma}S(\two{h}))\\
    =&(\one{h}\triangleright b)(\two{h}\triangleright \sigma^{-1}(S(\five{h}), \six{h}))\sigma(\three{h}, S(\four{h}))\#_{\sigma} 1\\
    =&h\triangleright b\#_{\sigma} 1,
\end{align*}
where the last step uses Proposition \ref{lemma. unital twist}. Hence, the twisted action is unchanged.
And the deformed 2-cocycle $\sigma_{\chi}$ is given by
\begin{align*}
    \sigma_{\chi}(h, g)=&j_{\chi}(\one{h})\cdot_{\chi}j_{\chi}(\one{g})\cdot_{\chi}j^{-1}_{\chi}(\two{h}\cdot_{\chi}\two{g})\\
    =&(j(\one{h})j(\one{g})\chi^{-1}(\two{h}, \two{g}))\cdot_{\chi}j^{-1}_{\chi}(\chi(\three{h}, \three{g})\four{h}\four{g}\chi^{-1}(\five{h}, \five{g}))\\
    =&(j(\one{h})j(\one{g}))\cdot_{\chi}j^{-1}_{\chi}(\two{h}\two{g})\chi^{-1}(\three{h}, \three{g})\\
    =&(j(\one{h})j(\one{g}))\cdot_{\chi}(\chi(\two{h}\two{g}, S(\three{h}\three{g}))j^{-1}(\four{h}\four{g}))\chi^{-1}(\five{h}, \five{g})\\
    =&j(\one{h})j(\one{g})j^{-1}(\six{h}\six{g})\chi^{-1}(\two{h}\two{g}, S(\five{h}\five{g}))\chi(\three{h}\three{g}, S(\four{h}\four{g}))\chi^{-1}(\seven{h}, \seven{g})\\
    =&j(\one{h})j(\one{g})j^{-1}(\two{h}\two{g})\chi^{-1}(\three{h}, \three{g})\\
    =&\sigma(\one{h}, \one{g})\chi^{-1}(\two{h}, \two{g})
\end{align*}
for any $h, g\in H$ and $b\in B$.

For the last part, if we we gauge transform by $u:H\to B$ then $(\sigma^u)_\chi = (\sigma_\chi)^{u}$. Indeed,
\begin{align*}
    (\sigma^u)_\chi(h, g)=&u^{-1}(\one{h})(\two{h}\triangleright u^{-1}(\one{g}))\sigma(\three{h}, \two{g})u(\four{h}\three{g})\chi^{-1}(\five{h}, \four{g})\\
    =&u^{-1}(\one{h})(\two{h}\triangleright u^{-1}(\one{g}))\sigma(\three{h}, \two{g})\chi^{-1}(\four{h}, \five{g})\chi(\six{h}, \five{g})u(\seven{h}\six{g})\chi^{-1}(\eight{h}, \seven{g})\\
    =&u^{-1}(\one{h})(\two{h}\triangleright u^{-1}(\one{g}))\sigma(\three{h}, \two{g})\chi^{-1}(\four{h}, \five{g})u(\five{h}\cdot_{\chi}\four{g})\\
    =&(\sigma_{\chi})^{u}(h,g),
\end{align*}
so that the equivalence classes are in bijective correspondence under the cotwist.
\end{proof}

\subsection{Cocycles of associative type}

\begin{defi}\label{def: assoc type}
We say that a crossed product $B\#_{\sigma}H$ is of \textup{associative type} or  $(\la,\sigma)$ is a cocycle of associative type, if $B$ is an $H$-module algebra under $\la$.
\end{defi}
It is easy to see that this happens iff
\begin{align}\label{eqn. assoc type}
    \sigma(\one{h}, \one{g})((\two{h}\two{g})\triangleright b)=((\one{h}\one{g})\triangleright b) \sigma(\two{h}, \two{g}),
\end{align}
for all $b\in B$, $h, g\in H$. Clearly, in the associative type case,  $B$ is an $H$-module algebra under the same left action, and $B\#H$ is well defined. A lemma needed later is

\begin{lem}\label{lem: associative type} Let $B\#_{\sigma}H$ be of associative type. Then
\begin{align*}
    (S(\three{h})\triangleright b)\sigma^{-1}(S(\one{h}), \two{h})=\sigma^{-1}(S(\two{h}), \three{h})(S(\one{h})\triangleright b)
\end{align*}
\end{lem}
\begin{proof} Since $B\#_{\sigma}H$ is of associative type, we have
\begin{align*}
    b\sigma^{-1}(S(\one{h}), \two{h})=((S(\two{h})\three{h})\triangleright b)\sigma^{-1}(S(\one{h}), \four{h})=\sigma^{-1}(S(\two{h}), \three{h})(S(\one{h})\four{h})\triangleright b).
\end{align*}
Then
\begin{align*}
    (S(\three{h})\triangleright b)\sigma^{-1}(S(\one{h}), \two{h})=&\sigma^{-1}(S(\two{h}), \three{h})((S(\one{h})\four{h})\triangleright (S(\five{h})\triangleright b))\\
    =&\sigma^{-1}(S(\two{h}), \three{h})(S(\one{h})\triangleright b).
\end{align*}
\end{proof}

On the other hand, if $B\#_{\sigma}H$ is of associative type then $B\#_{\sigma^u}H$ need not be. We say that a cleft extension is of associative type if its underlying $(\la,\sigma)$ as in Section~\ref{seccocy} obeys Definition~\ref{def: assoc type}. Let $Z_{as}^2(H,B)$ be the set of such action-cocycle pairs for a fixed Hopf algebra $H$ and algebra $B$.

\begin{prop} \label{prop. Has}  Cleft extensions of associative type are classified by $\CH_{as}(H,B)$ defined as $Z_{as}^2(H,B)$ modulo transformation $(\la,\sigma)$ to $(\la^u,\sigma^u)$ where $u:H\to B$ is convolution invertible and obeys
 \begin{align*}
    u^{-1}(\one{h})(\two{h}\triangleright (u^{-1}(\one{g})(\two{g}\triangleright b)u(\three{g})))u(\three{h})=u^{-1}(\one{h}\one{g})((\two{h}\two{g})\triangleright b)u(\three{h}\three{g}),
\end{align*}
for any $h, g\in H$ and $b\in B$.
\end{prop}
\begin{proof} Any two equivalent cleft extensions are related by some convolution-invertible $u$ as recalled in  Section~\ref{seccocy}. Not all of them will preserve associative type, the condition for this being $h\triangleright^u(g\triangleright^u b)=(hg)\triangleright^u b$ which is equivalent to
 the condition stated  using the relation (\ref{equ. equivalent corssed product condition 1}). It follows, but can be checked explicitly, that if $(\la,\sigma)\sim(\la^u,\sigma^u)\sim ((\la^u)^v,(\sigma^u)^v)$ so that $v$ obeys the stated condition with respect to $\la^u$, then the convolution product $u*v$ obeys the stated conditions and $((\la^u)^v,(\sigma^u)^v)=(\la^{u*v},\sigma^{u*v})$. Similarly for $u^{-1}$ to go back from $(\la^u,\sigma^u)$.
\end{proof}

\begin{exa}\label{ex inner} Let $H$ be a Hopf algebra, $B$ an $H$-module algebra of an inner form $h\la b= u(h\o)b u^{-1}(h\t)$ where  $u:H\to B$ is convolution invertible and such that
\[ u^{-1}(g\o)u^{-1}(h\o)u(h\t g\t)\in Z(B)\]
 for all $h,g\in H$ (this is necessary and sufficient to have an action). Then $(\la,1\eps\tens\eps)$ (i.e. with trivial cocycle) is of associative type and cohomologous in $\CH_{as}(H,B)$ to $(\id\eps,\sigma^u)$ (i.e. with trivial action) where $\sigma^u(h,g)=u^{-1}(g\o)u^{-1}(h\o)u(h\t g\t)$. Conversely, if $B$ is an $H$-module algebra $(\la, \sigma=1\eps\tens\eps)$ and is cohomologous to a cocycle of the form $(\id\eps,\sigma)$ (i.e. for trivial action) then $\la$ is of this inner form. For the first part, the condition in Proposition~\ref{prop. Has} reduces to the assumed property of $u$ and for this combination to lie in the centre $Z(B)$ is exactly the condition (\ref{eqn. assoc type}) for $\sigma^u$ to be of associative type when the action is trivial. In the other direction, for $\la^u$ to end up trivial, $\la$ has to be of the inner form in view of (\ref{equ. equivalent corssed product condition 1}). For a very concrete example, one could let $H=\mathbb{C}\Bbb Z_2$ and $B$ an algebra with an automorphism $\alpha$. If $\alpha^2=\id$ then we make $B$ into an $H$-module algebra
 with the generator of $\Bbb Z_2$ acting by $\alpha$. This action is of the inner form if and only if $\alpha$ is an inner automorphism. \end{exa}

Note that one could further restrict our gauge transformations to obey
\begin{equation}\label{eqn: strong equiv}  u(h\o) (h\t\la b)= (h\o\la b) u(h\t)  \end{equation}
which is stronger than the condition in Proposition~\ref{prop. Has} and enforces that $\la^u=\la$. This would define a cohomology $\CH^2_\la(H,B)$ for cocycle cross products for a fixed $H$-module algebra $B$. This condition is, however, quite restrictive. Also note that our associative type construction should not be confused with lazy Sweedler cohomology in the terminology of \cite{BC} where $\sigma$ is $k$-valued.

 Similarly, not every cotwist preserves associative type:

\begin{prop}\label{prop. cocy as twist} A Drinfeld co-twist $\chi$ as in Proposition~\ref{prop. twist cleft Galois object} preserves associative type iff  $(\la,\chi)$ viewed as a cocycle with values in $k.1\in B$ is  of associative type.
\end{prop}
\begin{proof} Indeed,
\begin{align*}
    (h\cdot_{\chi}g)\triangleright_{\chi} b=\chi(\one{h}, \one{g})((\two{h}\two{g})\triangleright b)\chi^{-1}(\three{h}, \three{g})\neq h\triangleright_{\chi}(g\triangleright_{\chi} b)=h\la(g\la b)=(hg)\la b
\end{align*}
would need $(h\cdot_\chi g)\la= (hg)\la$ i.e. act the same on $B$, and we see that this happens if and only if
\[ \chi(\one{h}, \one{g})((\two{h}\two{g})\triangleright b)\chi^{-1}(\three{h}, \three{g})  =(hg)\la b \]
which we interpret as stated. \end{proof}

\section{Antipode and twisting of cocycle Hopf algebroids} \label{sec3}

We start with preliminaries on bialgebroids before proceeding to the our new results for certain `cocycle Hopf algebroids'  $B^e\#_\sigma H$. These are a special case of  Ehresmann-Schauenburg in Section~4 but of interest in their own right.

\subsection{Preliminaries on bialgebroids and Hopf algebroids}

Here we recall the basic definitions (cf. \cite{BW}, \cite{Boehm}). Let $B$ be a unital algebra over a field $k$.
A {\em $B$-ring}  just means a unital algebra in the monoidal category ${}_B\CM_B$ of $B$-bimodules. Likewise,  a  {\em $B$-coring} is a coalgebra in ${}_B\CM_B$. Morphisms of $B$-(co)rings map the (co)algebra structures but now in the category ${}_B\CM_B$.

In practice, specifying a unital $B$-ring $\cL$ is equivalent to specifying a unital algebra $\cL$ (over $k$) and an algebra map $\eta:B\to \cL$. Left and right multiplication in $\cL$ pull back to left and right $B$-actions as a bimodule (so $b.X.c=\eta(b)X\eta(c)$ for all $b,c\in B$ and $X\in \cL$) and the product descends to the product $\mu_B:\cL\tens_B\cL\to \cL$ with $\eta$ the unit map. Conversely, given
$\mu_B$ we can pull back to an associative product on $\cL$ with unit $\eta(1)$.

Now suppose that $s:B\to \cL$ and $t:B^{op}\to \cL$ are algebra maps with images that commute. Then $\eta(b\tens c)=s(b)t(c)$ is an algebra map $\eta: B^e\to \cL$, where $B^e=B\tens B^{op}$, and is equivalent to making $\cL$ a $B^e$-ring. The left $B^e$-action part of this is equivalent to a $B$-bimodule structure
\begin{equation}\label{eq:rbgd.bimod}
b.X.c= s(b) t(c)X
\end{equation}
for all $b,c\in B$ and $X\in \cL$.

\begin{defi}\label{def:right.bgd} Let $B$ be a unital algebra. A left $B$-bialgebroid is an algebra $\cL$, commuting (`source' and `target') algebra maps $s:B\to \cL$ and $t:B^{op}\to \cL$ (and hence a $B^e$-ring) and a $B$-coring for the bimodule structure (\ref{eq:rbgd.bimod}) which is compatible in the sense:
\begin{itemize}
\item[(i)] The coproduct $\Delta$ corestricts to an algebra map  $\cL\to \cL\times_B \cL$ where
\begin{equation*} \cL\times_{B} \cL :=\{\ \sum_i X_i \ot_{B} Y_i\ |\ \sum_i X_it(b) \ot_{B} Y_i=
\sum_i X_i \ot_{B}  Y_i  s(b),\quad \forall b\in B\ \}\subseteq \cL\tens_B\cL,
\end{equation*}
 is an algebra via factorwise multiplication.
\item[(ii)] The counit $\epsilon$ is a `left character' in the sense
\begin{equation*}\epsilon(1_{\cL})=1_{B},\quad \epsilon( s(b)X)=b\epsilon(X), \quad \epsilon(Xs(\epsilon(Y)))=\epsilon(XY)=\epsilon(Xt (\epsilon(Y)))\end{equation*}
for all $X,Y\in \cL$ and $b\in B$.
\end{itemize}
\end{defi}

Morphisms between left $B$-bialgebroids are $B$-coring maps which are also algebra maps.

\begin{defi}\label{defHopf}
A left bialgebroid $\cL$ is a left Hopf algebroid (\cite{schau1}, Thm and Def 3.5.) if
\[\lambda: \cL\ot_{B^{op}}\cL\to \cL\ot_{B}\cL,\quad
    \lambda(X\ot_{B^{op}} Y)=\one{X}\ot_{B}\two{X}Y\]
is invertible, where $\tens_{B^{op}}$ is induced by $t$ (so $Xt(b)\ot_{B^{op}}Y=X\ot_{B^{op}}t(b)Y$, for all $X, Y\in \cL$ and $b\in B$) while the $\tens_B$ is the standard one (\ref{eq:rbgd.bimod}) (so $t(b)X\ot_{B}Y=X\ot_{B}s(b)Y$).
\end{defi}
In the following, by Hopf algebroid we will always mean a left Hopf algebroid. If $B=k$ then this reduces to the map $\cL\tens\cL\to \cL\tens\cL$ given by $h\tens g\mapsto h\o\tens h\t g$ which for a usual Hopf algebra has inverse $h\tens g\mapsto h\o\tens S(h\t)g$ if there is an antipode.

\begin{lem}\label{lemS}
Let $\cL$ be a left bialgebroid. If there is an invertible linear map $S:\cL\to\cL$, such that
\begin{itemize}
    \item [(1)] $S(t(b)X)=S(X)s(b)$,  \qquad $S^{-1}(s(b)X)=S^{-1}(X)t(b)$
    \item[(2)] $\one{S^{-1}(\two{X})}\ot_{B}\two{S^{-1}(\two{X})}\one{X}=S^{-1}(X)\ot_{B}1$,
    \item[(3)] $\one{S(\one{X})}\two{X}\ot_{B}\two{S(\one{X})}=1\ot_{B}S(X)$,
\end{itemize}
for all $X, Y\in \cL$ and $b\in B$ then $\cL$ is a left Hopf algebroid with $\lambda^{-1}$ given by
\begin{align*}
    \lambda^{-1}(X\ot_{B}Y):=S^{-1}(\two{S(X)})\ot_{B^{op}}\one{S(X)}Y.
\end{align*}
We call such $S$ a  {\em (left) antipode} for $\CL$.
\end{lem}
\begin{proof}
First we check conditions (2), (3) and that the stated $\lambda^{-1}$ is well defined.
For (3),
\begin{align*}
    \one{S(t(b)\one{X})}\two{X}\ot_{B}\two{S(t(b)\one{X})}=\one{S(\one{X})}s(b)\two{X}\ot_{B}\two{S(\one{X})}.
\end{align*}
Similar for (2). For $\lambda^{-1}$,
\begin{align*}
    S^{-1}(s(b)\two{S(X)})\ot_{B^{op}}\one{S(X)}Y=&S^{-1}(\two{S(X)})t(b)\ot_{B^{op}}\one{S(X)}Y\\
    =&S^{-1}(\two{S(X)})\ot_{B^{op}}t(b)\one{S(X)}Y.
\end{align*}
It remains to check that $\lambda^{-1}$ is the inverse of $\lambda$:
\begin{align*}
\lambda(\lambda^{-1}(X\ot_{B}Y))=&\lambda(S^{-1}(\two{S(X)})\ot_{B^{op}}\one{S(X)}Y)\\
=&\one{S^{-1}(\two{S(X)})}\ot_{B}\two{S^{-1}(\two{S(X)})}\one{S(X)}Y\\
=&X\ot_{B}Y.\\
    \lambda^{-1}(\lambda(X\ot_{B^{op}}Y))=&\lambda^{-1}(\one{X}\ot_{B}\two{X}Y)\\
    =&S^{-1}(\two{S(\one{X})})\ot_{B^{op}}\one{S(\one{X})}\two{X}Y\\
    =&X\ot_{B^{op}}Y
\end{align*}
for all $X,Y\in \CL$.
\end{proof}

Note that we are not claiming that $S$ here is unique. In what follows, we adopt the shorthand
\begin{equation}\label{X+-} X_{+}\ot_{B^{op}}X_{-}:=\lambda^{-1}(X\ot_{B}1).\end{equation}

\begin{rem} (1) If $B=k$ then Definition~\ref{defHopf}  reduces to the usual definition of a Hopf algebra.  We first observe that
\begin{itemize}
    \item [(i)] $\one{h}{}_{+}\ot\one{h}{}_{-}\two{h}=h\ot 1$.
    \item[(ii)] $\one{h_{+}}\ot \two{h_{+}}h_{-}=h\ot 1$
    \item[(iii)] $\one{h_{+}}\ot \two{h_{+}}\ot h_{-}=\one{h}\ot \two{h}{}_{+}\ot \two{h}{}_{-}$
    \item[(iv)]$h_{+}h_{-}=\eps(h)$
\end{itemize}
where  (i) and (ii) are from applying $\lambda^{-1}\circ \lambda$ and $\lambda\circ \lambda^{-1}$ on $h\ot 1$. Then (iv) is the result by applying $\eps\ot \id$ to (ii). For (iii), apply $\id\ot \lambda$ on the right hand side of (iii) to get $\one{h}\ot\two{h}\ot 1$. But applying $\id\ot \lambda$ on the left hand side of (iii) we get
\begin{align*}
    \one{h_{+}}\ot \two{h_{+}}\ot \three{h_{+}}h_{-}=\one{\one{h_{+}}}\ot \two{\one{h_{+}}}\ot \two{h_{+}}h_{-}=\one{h}\ot\two{h}\ot 1
\end{align*}
also, where the second steps uses (ii). Since $\id\ot\lambda$ is a bijective map, we have (iii).
Now set $S=\eps(h_{+})h_{-}$ and we check
\begin{align*}
    S(\one{h})\two{h}=\eps(\one{h}{}_{+})\one{h}{}_{-}\two{h}=\eps(h),
\end{align*}
where we use (i) in the second step. On the other hand,
\begin{align*}
    \one{h}S(\two{h})=\one{h}\eps(\two{h}{}_{+})\two{h}{}_{-}=\one{h_{+}}\eps(\two{h_{+}})h_{-}=h_{+}h_{-}=\eps(h),
\end{align*}
where the second step uses (iii) and the last step uses (iv).

(2) It follows that if $B=k$ then the conditions in Lemma \ref{lemS} is equivalent to a usual Hopf algebra with invertible antipode. This is because of these conditions hold then $S$ is assumed invertible and moreover, Definition~\ref{defHopf} applies so we have a Hopf algebra by part (1). Conversely, if we have a Hopf algebra with invertible antipode,  (2) and (3) in Lemma \ref{lemS} hold using that $S,S^{-1}$ are anticoalgebra maps.
\end{rem}

We will need some identities, and we recall from \cite[Prop.~3.7]{schau1} that for a left Hopf algebroid,
\begin{align}
    \one{X_{+}}\ot_{B}\two{X_{+}}X_{-}&=X\ot_{B}1\label{equ. inverse lamda 1};\\
    \one{X}{}_{+}\ot_{B^{op}}\one{X}{}_{-}\two{X}&=X\ot_{B^{op}}1\label{equ. inverse lamda 2};\\
    (XY)_{+}\ot_{B^{op}}(XY)_{-}&=X_{+}Y_{+}\ot_{B^{op}}Y_{-}X_{-}\label{equ. inverse lamda 3};\\
    1_{+}\ot_{B^{op}}1_{-}&=1\ot_{B^{op}}1\label{equ. inverse lamda 4};\\
    \one{X_{+}}\ot_{B}\two{X_{+}}\ot_{B^{op}}X_{-}&=\one{X}\ot_{B}\two{X}{}_{+}\ot_{B^{op}}\two{X}{}_{-}\label{equ. inverse lamda 5};\\
    X_{+}\ot_{B^{op}}\one{X_{-}}\ot_{B}\two{X_{-}}&=X_{++}\ot_{B^{op}}X_{-}\ot_{B}X_{+-}\label{equ. inverse lamda 6};\\
    X&=X_{+}t(\eps(X_{-}))\label{equ. inverse lamda 7};\\
    X_{+}X_{-}&=s(\eps(X))\label{equ. inverse lamda 8}.
\end{align}

Since $\lambda(s(b)\ot_{B^{op}}1)=s(b)\ot_{B} 1$, we have
\begin{align*}
    &s(b)_{+}\ot_{B^{op}} s(b)_{-}=s(b)\ot_{B^{op}} 1\\
    %&t(b)_{+}\ot_{B^{op}} t(b)_{-}=1\ot_{B^{op}} t(b)=t(b)\ot_{B^{op}} 1
\end{align*}
As a result, from (\ref{equ. inverse lamda 3}), we get
\begin{align}
    &(Xs(b))_{+}\ot_{B^{op}} (Xs(b))_{-}=X_{+}s(b)\ot_{B^{op}} X_{-},\label{equ. source and target map with lambda inv 1}\\
    %&(Xt(b))_{+}\ot_{B^{op}} (Xt(b))_{-}=X_{+}\ot_{B^{op}} t(b)X_{-}=X_{+}t(b)\ot_{B^{op}} X_{-}\label{equ. source and target map with lambda inv 2}\\
    &(s(b)X)_{+}\ot_{B^{op}} (s(b)X)_{-}=s(b)X_{+}\ot_{B^{op}} X_{-}.\label{equ. source and target map with lambda inv 3}
    %&(t(b)X)_{+}\ot_{B^{op}} (t(b)X)_{-}=X_{+}\ot_{B^{op}} X_{-}t(b).\label{equ. source and target map with lambda inv 4}
\end{align}
Since $\lambda^{-1}(t(b)X\ot_{B}1)=\lambda^{-1}(X\ot_{B}s(b))=X_{+}\ot_{B^{op}}X_{-}s(b)$, we have
\begin{align}
   (t(b)X)_{+}\ot_{B^{op}} (t(b)X)_{-}=X_{+}\ot_{B^{op}}X_{-}s(b)\label{equ. source and target map with lambda inv 5}.
\end{align}

Moreover, since $\lambda(X_{+}\ot_{B^{op}}s(b)X_{-})=\one{X_{+}}\ot_{B}\two{X_{+}}s(b)X_{-}=\one{X_{+}}t(b)\ot_{B}\two{X_{+}}X_{-}=Xt(b)\ot_{B}1$, we have

\begin{align}
    (Xt(b))_{+}\ot_{B^{op}}(Xt(b))_{-}=X_{+}\ot_{B^{op}}s(b)X_{-}\label{equ. source and target map with lambda inv 6}.
\end{align}

\subsection{Twisted Hopf algebroids}

Given a left $B$-bialgebroid $\cL$,   there is an algebra structure on ${}_{B}\Hom_{B}(\cL\otimes_{B^e}\cL, B)$ with the (convolution) product and unit given by
\begin{align}
    f\star g(X, Y):=f(\one{X}, \one{Y})g(\two{X}, \two{Y}),\quad \tilde{\epsilon}(X\tens Y)= \epsilon(XY)
\end{align}
for all $X, Y\in \cL$ and $f, g\in {}_{B}\Hom_{B}(\cL\otimes_{B^e}\cL, B)$. The $B$-bimodule structure on $\cL\otimes_{B^e}\cL$ is just the left $B^e$ action (so $b.(X\otimes Y). c=s(b)t(c)X\otimes Y$  for all $b, c\in B$) as induced by the algebra map $\eta: B^e\to \cL$.

\begin{defi} \cite{Boehm}
Let $\cL$ be a left $B$-bialgebroid. A \textup{normalised left 2-cocycle} on
$\cL$ is a convolution invertible element $\varsigma\in {}_{B}\Hom_{B}(\cL\otimes_{B^e}\cL, B)$ such that
\[ \varsigma(X, s(\varsigma(\one{Y}, \one{Z}))\two{Y}\two{Z})=\varsigma(s(\varsigma(\one{X}, \one{Y}))\two{X}\two{Y}, Z),\quad \varsigma(1_{\cL}, X)=\epsilon(X)=\varsigma(X, 1_{\cL})\]
for all $X, Y, Z\in \cL$.

Similarly, a \textup{normalised right 2-cocycle} on
$\cL$ is a convolution invertible element $\xi\in {}_{B}\Hom_{B}(\cL\otimes_{B^e}\cL, B)$, such that
\[ \xi(X, t(\xi(\two{Y}, \two{Z}))\one{Y}\one{Z})=\xi(t(\xi(\two{X}, \two{Y}))\one{X}\one{Y}, Z),\quad \xi(1_{\cL}, X)=\epsilon(X)=\xi(X, 1_{\cL})\]
for all $X, Y, Z\in \cL$.
A normalised left 2-cocycle is called an \textup{invertible normalised 2-cocycle}, if its convolution inverse is a normalised right 2-cocycle.
\end{defi}

If $B=k$ the field then a bialgebroid reduces to a usual bialgebra and $\varsigma$ reduces to a usual Drinfeld cotwist cocycle in the sense of \cite{Ma:book}. However, unlike a 2-cocycle on a bialgebra, for any normalised left 2-cocycle on a bialgebroid, its convolution inverse is not automatically a normalised right 2-cocycle.

\begin{lem}\label{lemma. 2 cocycle and its inverse}
Let $\cL$ be a left $B$-bialgebroid and $\varsigma\in {}_{B}\Hom_{B}(\cL\otimes_{B^e}\cL, B)$  an invertible normalised  2-cocycle with inverse $\varsigma^{-1}$. Then
\begin{align}
    \varsigma(\one{X}, \one{Y}\one{Z})\varsigma^{-1}(\two{X}\two{Y}, \two{Z})&=\varsigma(Xs(\varsigma^{-1}(\one{Y}, Z)), \two{Y}),\label{equ. 2 cocycle and its inverse1}\\
    \varsigma(\one{X}\one{Y}, \one{Z})\varsigma^{-1}(\two{X}, \two{Y} \two{Z})&=\varsigma^{-1}(X, t(\varsigma(\two{Y}, Z))\one{Y}).\label{equ. 2 cocycle and its inverse2}
\end{align}
\end{lem}

\begin{proof}
Since the inverse of $\varsigma$ satisfies
\[ \varsigma^{-1}(X, t(\varsigma^{-1}(\two{Y}, \two{Z}))\one{Y}\one{Z})=\varsigma^{-1}(t(\varsigma^{-1}(\two{X}, \two{Y}))\one{X}\one{Y}, Z),\]
we can write
\begin{align*}
  \varsigma(\one{X}, \one{Y}&\one{Z})\varsigma^{-1}(\two{X}, t(\varsigma^{-1}(\three{Y}, \three{Z}))\two{Y}\two{Z})\varsigma(\three{X}, \four{Y})\\
=&\varsigma(\one{X}, \one{Y}\one{Z})\varsigma^{-1}(t(\varsigma^{-1}(\three{X}, \three{Y}))\two{X}\two{Y}, \two{Z})\varsigma(\four{X}, \four{Y}).
\end{align*}
The left hand side is equal to
\begin{align*}
    \varsigma(\one{X}, \one{Y}&\one{Z})\varsigma^{-1}(\two{X}, t(\varsigma^{-1}(\three{Y}, \three{Z}))\two{Y}\two{Z})\varsigma(\three{X}, \four{Y})\\
    =&\varsigma(\one{X}, \one{Y}\one{Z})\varsigma^{-1}(\two{X}, \two{Y}\two{Z})\varsigma(\three{X}s(\varsigma^{-1}(\three{Y}, \three{Z})), \four{Y})\\
    =&\epsilon(\one{X} \one{Y}\one{Z})\varsigma(\two{X}s(\varsigma^{-1}(\two{Y}, \two{Z})), \three{Y})\\
    =&\epsilon(\one{X} \one{Y}t(\eps(\one{Z})))\varsigma(\two{X}s(\varsigma^{-1}(\two{Y}, \two{Z})), \three{Y})\\
    =&\epsilon(\one{X} \one{Y})\varsigma(\two{X}s(\varsigma^{-1}(\two{Y}s(\eps(\one{Z})), \two{Z})), \three{Y})\\
    =&\epsilon(\one{X} t(\eps(\one{Y})))\varsigma(\two{X}s(\varsigma^{-1}(\two{Y},Z)), \three{Y})\\
    =&\epsilon(\one{X} )\varsigma(\two{X}s(\eps(\one{Y}))s(\varsigma^{-1}(\two{Y},Z)), \three{Y})\\
    =&\epsilon(\one{X} )\varsigma(\two{X}s(\varsigma^{-1}(\eps(\one{Y})\two{Y},Z)), \three{Y})\\
    =&\epsilon(\one{X} )\varsigma(\two{X}s(\varsigma^{-1}(\one{Y},Z)), \two{Y})\\
    =&\varsigma(Xs(\varsigma^{-1}(\one{Y},Z)), \two{Y}).
\end{align*}
The right hand side is equal to
\begin{align*}
    \varsigma(\one{X}, \one{Y}&\one{Z})\varsigma^{-1}(t(\varsigma^{-1}(\three{X}, \three{Y}))\two{X}\two{Y}, \two{Z})\varsigma(\four{X}, \four{Y})\\
    =&\varsigma(\one{X}, \one{Y}\one{Z})\varsigma^{-1}(\two{X}\two{Y}, \two{Z})\varsigma^{-1}(\three{X}, \three{Y})\varsigma(\four{X}, \four{Y})\\
    =&\varsigma(\one{X}, \one{Y}\one{Z})\varsigma^{-1}(\two{X}\two{Y}, \two{Z})\eps(\three{X}\three{Y})\\
    =&\varsigma(\one{X}, \one{Y}\one{Z})\varsigma^{-1}(\two{X}\two{Y}, \two{Z}),
\end{align*}
proving (\ref{equ. 2 cocycle and its inverse1}). Similarly for (\ref{equ. 2 cocycle and its inverse2}).
\end{proof}

Using such cotwists, it known from \cite{Boehm} that  a $B$-coring $\cL$ with the twisted product
\begin{align}\label{twistprod}
    X\cdot_{\varsigma} Y:=s(\varsigma(\one{X}, \one{Y}))t(\varsigma^{-1}(\three{X}, \three{Y}))\two{X}\two{Y},
\end{align}
for all $X, Y\in \cL$, and the original coproduct, counit, $s, t$ constitute  a left bilgebroid $\cL^\varsigma$.

\begin{thm}\label{thm. 2 cocycle twist}
Let $\cL$ be a left Hopf algebroid and $\varsigma\in {}_{B}\Hom_{B}(\cL\otimes_{B^e}\cL, B)$ be an invertible normalised 2-cocycle with inverse $\varsigma^{-1}$. Then the left bialgebroid $\cL^\varsigma$ with twisted product (\ref{twistprod}) is a left Hopf algebroid, with
\begin{align*}
   ( \lambda^{\varsigma})^{-1}(X\ot_{B} Y)=t(\varsigma(\two{\one{X}{}_{+}}, \three{\one{X}{}_{-}}))\one{\one{X}{}_{+}}\ot_{B^{op}}s(\varsigma^{-1}(\one{\one{X}{}_{-}}, \two{X}))\two{\one{X}{}_{-}}\rosso{\cdot_{\varsigma}}Y.
\end{align*}
\end{thm}
\begin{proof}
We need to show that $\lambda^{\varsigma}: X\ot_{B^{op}}Y\mapsto \one{X}\ot_{B}\two{X}\cdot_{\varsigma}Y$ is an invertible map with inverse of $\lambda^{\varsigma}$ as stated.  For this formula to be well defined, we  need to show that it factors through $\one{X}\tens_{B}\two{X}$,
\begin{align*}
    t(\varsigma&(\two{(t(b)\one{X}){}_{+}}, \three{(t(b)\one{X}){}_{-}}))\one{(t(b)\one{X}){}_{+}}\ot_{B^{op}}s(\varsigma^{-1}(\one{(t(b)\one{X}){}_{-}}, \two{X}))\two{(t(b)\one{X}){}_{-}}\\
    %=&t(\varsigma(\two{\one{X}{}_{+}}, \three{(\one{X}{}_{-}t(b))}))\one{\one{X}{}_{+}}\ot_{B^{op}}s(\varsigma^{-1}(\one{(\one{X}{}_{-}t(b))}, \two{X}))\two{(\one{X}{}_{-}t(b))}\\
    =&t(\varsigma(\two{\one{X}{}_{+}}, \three{(\one{X}{}_{-}s(b))}))\one{\one{X}{}_{+}}\ot_{B^{op}}s(\varsigma^{-1}(\one{(\one{X}{}_{-}s(b))}, \two{X}))\two{(\one{X}{}_{-}s(b))}\\
    =&t(\varsigma(\two{\one{X}{}_{+}}, \three{\one{X}{}_{-}}))\one{\one{X}{}_{+}}\ot_{B^{op}}s(\varsigma^{-1}(\one{\one{X}{}_{-}}, s(b)\two{X}))\two{\one{X}{}_{-}},
\end{align*}
where the 1st step uses (\ref{equ. source and target map with lambda inv 5}). Similar checks for other cases such as $\one{X}{}_{+}\tens_{B^{op}}\one{X}{}_{-}$ etc.  as well as the $X\tens_BY$ are easier and left to the reader.  Next, it is sufficient to check $(\lambda^{\varsigma})^{-1}\circ\lambda^{\varsigma}$ and $\lambda^{\varsigma}\circ(\lambda^{\varsigma})^{-1}$ on $X\ot_{B^{op}}1$ and $X\ot_{B} 1$. On one side, we have
\begin{align*}
 (\lambda^{\varsigma})^{-1}&\circ\lambda^{\varsigma}(X\ot_{B^{op}}1)\\
   =&t(\varsigma(\two{\one{X}{}_{+}}, \three{\one{X}{}_{-}}))\one{\one{X}{}_{+}}\ot_{B^{op}}(s(\varsigma(\one{\one{X}{}_{-}}, \two{X}))\two{\one{X}{}_{-}})\cdot_{\varsigma}\three{X}\\
   =&t(\varsigma(\two{\one{X}{}_{+}}, \five{\one{X}{}_{-}}))\one{\one{X}{}_{+}}\\
& \quad  \ot_{B^{op}}s(\varsigma(s(\varsigma^{-1}(\one{\one{X}{}_{-}}, \two{X}))\two{\one{X}{}_{-}}, \three{X}))t(\varsigma^{-1}(\four{\one{X}{}_{-}}, \five{X}))\three{\one{X}{}_{-}}\four{X}\\
   =&t(\varsigma(\two{\one{X}{}_{+}}, \five{\one{X}{}_{-}}))\one{\one{X}{}_{+}}\\
  &\quad  \ot_{B^{op}}s(\varsigma^{-1}(\one{\one{X}{}_{-}}, \two{X})\varsigma(\two{\one{X}{}_{-}}, \three{X}))t(\varsigma^{-1}(\four{\one{X}{}_{-}}, \five{X}))\three{\one{X}{}_{-}}\four{X}\\
   =&t(\varsigma(\two{\one{X}{}_{+}{}_{+}}, \two{\one{X}{}_{+}{}_{-}}))\one{\one{X}{}_{+}{}_{+}}\ot_{B^{op}}t(\varsigma^{-1}(\one{\one{X}{}_{+}{}_{-}}, \three{X}))\one{X}{}_{-}\two{X}\\
   =&t(\varsigma(\two{\one{X}{}_{+}}, \two{\one{X}{}_{-}}))\one{\one{X}{}_{+}}\ot_{B^{op}}t(\varsigma^{-1}(\one{\one{X}{}_{-}}, \two{X}))\\
   =&t(\varsigma(\two{\one{X}{}_{+}}, \two{\one{X}{}_{-}}))\one{\one{X}{}_{+}}t(\varsigma^{-1}(\one{\one{X}{}_{-}}, \two{X}))\ot_{B^{op}} 1\\
   =&t(\varsigma(\two{\one{X}{}_{+}}s(\varsigma^{-1}(\one{\one{X}{}_{-}}, \two{X})), \two{\one{X}{}_{-}}))\one{\one{X}{}_{+}}\ot_{B^{op}} 1\\
   =&t(\varsigma(\two{\one{X}{}_{+}}, \one{\one{X}{}_{-}}\two{X})\varsigma^{-1}(\three{\one{X}{}_{+}}\two{\one{X}{}_{-}}, \three{X}))\one{\one{X}{}_{+}}\ot_{B^{op}} 1\\
   =&t(\varsigma(\two{\one{X}{}_{+}{}_{+}}, \one{X}{}_{-}\two{X})\varsigma^{-1}(\three{\one{X}{}_{+}{}_{+}}\one{X}{}_{+}{}_{-}, \three{X}))\one{\one{X}{}_{+}{}_{+}}\ot_{B^{op}} 1\\
   =&t(\eps(\two{\one{X}{}_{+}})\varsigma^{-1}(\three{\one{X}{}_{+}}\one{X}{}_{-}, \two{X}))\one{\one{X}{}_{+}}\ot_{B^{op}} 1\\
   =&t(\varsigma^{-1}(\two{\one{X}{}_{+}}\one{X}{}_{-}, \two{X}))\one{\one{X}{}_{+}}\ot_{B^{op}} 1\\
   =&t(\eps(\two{X}))\one{X}\ot_{B^{op}} 1\\
   =&X\ot_{B^{op}} 1,
\end{align*}
where the 4th and 9th steps use (\ref{equ. inverse lamda 6}), the 5th and 10th steps use (\ref{equ. inverse lamda 2}), the 8th step use (\ref{equ. 2 cocycle and its inverse1}), the 12th step uses (\ref{equ. inverse lamda 1}). On the other side, we have
\begin{align*}
    \lambda^{\varsigma}\circ&(\lambda^{\varsigma})^{-1}(X\ot_{B}1)\\
    =&\one{\one{X}{}_{+}}\ot_{B}s(\varsigma(\two{\one{X}{}_{+}}, s(\varsigma^{-1}(\one{\one{X}{}_{-}}, \two{X}))\two{\one{X}{}_{-}}))\\
    &t(\varsigma^{-1}(t(\varsigma(\five{\one{X}{}_{+}}, \five{\one{X}{}_{-}}))\four{\one{X}{}_{+}}, \four{\one{X}{}_{-}}))\three{\one{X}{}_{+}}\three{\one{X}{}_{-}}\\
    =&\one{\one{X}{}_{+}}\ot_{B}s(\varsigma(\two{\one{X}{}_{+}}, s(\varsigma^{-1}(\one{\one{X}{}_{-}}, \two{X}))\two{\one{X}{}_{-}}))\three{\one{X}{}_{+}}\three{\one{X}{}_{-}}\\
    =&\one{\one{X}{}_{+}{}_{+}}\ot_{B}s(\varsigma(\two{\one{X}{}_{+}{}_{+}}, s(\varsigma^{-1}(\one{\one{X}{}_{-}}, \two{X}))\two{\one{X}{}_{-}}))\three{\one{X}{}_{+}{}_{+}}\one{X}{}_{+}{}_{-}\\
    =&\one{\one{X}{}_{+}}\ot_{B}s(\varsigma(\two{\one{X}{}_{+}}, s(\varsigma^{-1}(\one{\one{X}{}_{-}}, \two{X}))\two{\one{X}{}_{-}}))\three{\one{X}{}_{+}}{}_{+}\three{\one{X}{}_{+}}{}_{-}\\
    =&\one{\one{X}{}_{+}}\ot_{B}s(\varsigma(\two{\one{X}{}_{+}}, s(\varsigma^{-1}(\one{\one{X}{}_{-}}, \two{X}))\two{\one{X}{}_{-}}))s(\eps(\three{\one{X}{}_{+}}{}))\\
    =&\one{\one{X}{}_{+}}\ot_{B}s(\varsigma(\two{\one{X}{}_{+}}, s(\varsigma^{-1}(\one{\one{X}{}_{-}}, \two{X}))\two{\one{X}{}_{-}}))\\
    =&\one{\one{X}{}_{+}}\ot_{B}s(\varsigma(\two{\one{X}{}_{+}}, \one{\one{X}{}_{-}}\two{X})\varsigma^{-1}(\three{\one{X}{}_{+}}\two{\one{X}{}_{-}}, \three{X}))\\
    =&\one{\one{X}{}_{+}{}_{+}}\ot_{B}s(\varsigma(\two{\one{X}{}_{+}{}_{+}}, \one{X}{}_{-}\two{X})\varsigma^{-1}(\three{\one{X}{}_{+}{}_{+}}\one{X}{}_{+}{}_{-}, \three{X}))\\
    =&\one{\one{X}{}_{+}}\ot_{B}s(\eps(\two{\one{X}{}_{+}}))\varsigma^{-1}(\three{\one{X}{}_{+}}\one{X}{}_{-}, \two{X}))\\
    =&\one{\one{X}{}_{+}}\ot_{B}s(\varsigma^{-1}(\two{\one{X}{}_{+}}\one{X}{}_{-}, \two{X}))\\
    =&\one{X}\ot_{B} s(\eps(\two{X}))\\
    =&t(\eps(\two{X}))\one{X}\ot_{B} 1\\
    =&X\ot_{B} 1
\end{align*}
where the 3rd and 8th steps use (\ref{equ. inverse lamda 6}), the 4th step uses (\ref{equ. inverse lamda 5}) , the 5th step uses (\ref{equ. inverse lamda 8}), the 7th step use (\ref{equ. 2 cocycle and its inverse1}), the 11th step use (\ref{equ. inverse lamda 1}).

\end{proof}

\begin{rem} The right-handed version of the theorem also holds. Thus, for  an invertible
normalised 2-cocycle $\varsigma$ on a left bialgebroid $\cL$, if the map $\mu: \cL\ot^{B^{op}} \cL\to \cL\ot_{B} \cL$ given by
\begin{align*}
    \mu(X\ot^{B^{op}} Y):=\one{Y}X\ot_{B}\two{Y}
\end{align*}
is invertible, then $\mu^{\varsigma}: \cL^{\varsigma}\ot^{B^{op}} \cL^{\varsigma}\to \cL^{\varsigma}\ot_{B} \cL^{\varsigma}$ given by $\mu^{\varsigma}(X\ot^{B^{op}} Y):=\one{Y}\cdot_{\varsigma}X\ot_{B}\two{Y}$ is invertible as well. The balanced tensor product on the left side is given by  $X\ot^{B^{op}}Ys(b)=s(b)X\ot^{B^{op}}Y$, while the balanced tensor product on the right side is given by $X\ot_{B}s(b)Y=t(b)X\ot_{B}Y$. More explicitly, if
the image of $\mu^{-1}(1\ot_{B}X)$ is denoted by $X_{(-)}\ot^{B^{op}}X_{(+)}$, then
\begin{align*}
    (\mu^{\varsigma})^{-1}(1\ot_{B}X)=t(\varsigma(\three{\two{X}{}_{(-)}}, \one{X}))\two{\two{X}{}_{(-)}}\ot^{B^{op}}s(\varsigma^{-1}(\one{\two{X}{}_{(+)}}, \one{\two{X}{}_{(-)}}))\two{\two{X}{}_{(+)}}.
\end{align*}
\end{rem}

\subsection{Cocycle Hopf algebroids}

Here we directly build a Hopf algebroid from the data for any cleft cocycle extension.

\begin{prop}\label{prop cocy Hopf algebroid} Associated to a cocycle cross product $B\#_\sigma H$, we have a {\em cocycle  Hopf algebroid}
$B^e \#_\sigma  H$ built on $B^e\tens H$
with product
\[(b\tens b'\#_{\sigma} h)(c\tens c'\#_{\sigma} g)=b(\one{h}\triangleright c)\sigma(\two{h}, \one{g})\ot c' (S(\four{g})\triangleright  b')\sigma(S(\three{g}), S(\four{h}))\#_{\sigma} \three{h}\two{g}\]
and source and target maps
\[s(b)=b\ot1\#_{\sigma} 1, \quad t(b)=1\ot b\#_{\sigma} 1,\quad \eta(b\ot b')=s(b)t(b')=b\ot b'\#_{\sigma} 1.\]
The coproduct is
\[ \Delta (b\tens b'\#_{\sigma} h)=(b\tens \sigma^{-1}(S(\two{h}), \three{h})\#_{\sigma} \one{h})\tens_B (1\tens b'\#_{\sigma} \four{h}),\]
where the $B$-bimodule structure is
\[ c. (b\ot b'\#_{\sigma} h). c'=cb\ot b' (S(\two{h})\triangleright c')\#_{\sigma} \one{h}.\]
The counit is
\[\eps(b\ot b'\#_{\sigma} h)=b(\one{h}\triangleright b')\sigma(\two{h}, S(\three{h})).\]
The inverse of $\lambda$ is given by
\begin{align*}
    \lambda^{-1}(X\ot_{B}Y)=(b\ot 1\ot\one{h})\ot_{B^{op}}(b'\ot \sigma^{-1}(S^{2}(\four{h}), S(\three{h}))\#_{\sigma}S(\five{h}))Y.
\end{align*}
\end{prop}

\begin{proof} First, we show $B^{e}\#_{\sigma} H$ is a $B$-coring and $B^{e}$-ring. It is not hard to see that both the coproduct and and the counit are $B$-bimodule maps. Moreover,
\begin{align*}
    (\eps\ot_{B}\id)(\Delta(b\ot b'\#_{\sigma} h))=&b(\one{h}\triangleright\sigma^{-1}(S(\four{h}), \five{h}))\sigma(\two{h}, S(\three{h}))\ot b'\#_{\sigma}\six{h}=b\ot b'\#_{\sigma} h,
\end{align*}
and
\begin{align*}
    (\id\ot_{B}\eps)&(\Delta(b\ot b'\#_{\sigma} h))=b\ot \sigma^{-1}(S(\three{h}), \four{h})(S(\two{h})\triangleright \big((\five{h}\triangleright b')\sigma(\six{h}, S(\seven{h}))\big))\#_{\sigma}\one{h}\\
    =&b\ot (S(\four{h})\five{h})\triangleright b')\sigma^{-1}(S(\three{h}), \six{h})(S(\two{h})\triangleright\sigma(\seven{h}, S(\eight{h})))\#_{\sigma}\one{h}\\
    =&b\ot b' \sigma^{-1}(S(\five{h}), \six{h})\sigma(S(\four{h}), \seven{h})\sigma(S(\three{h})\eight{h}, S(\eleven{h}))\\
    &\sigma^{-1}(S(\two{h}), \nine{h}S(\ten{h}))\#_{\sigma}\one{h}\\
    =&b\ot b'\#_{\sigma} h,
\end{align*}
where we use Lemma \ref{lemma. unital twist}  repeatedly. Also,
\begin{align*}
    (\id\ot_{B}&\Delta)(\Delta(b\ot b'\#_{\sigma} h))\\
    =&(b\ot\sigma^{-1}(S(\two{h}), \three{h}))\#_{\sigma}\one{h})\ot_{B}(1\ot \sigma^{-1}(S(\five{h}), \six{h})\#_{\sigma}\four{h})\ot_{B}(1\ot b'\#_{\sigma} \seven{h})\\
    =&(\Delta\ot_{B}\id)(\Delta(b\ot b'\#_{\sigma} h)).
\end{align*}
For the $B^{e}$ structure, it is easy to see that $\eta$ is an algebra map. It is a direct computation that the product is associative, since $\sigma$ is a 2-cocycle. Next, we check that the image of coproduct belongs to the Takeuchi product. Let $X=b\ot b'\ot h$ and $d\in B$. Then
\begin{align*}
    \one{X}\ot_{B}\two{X}s(d)=&(b\tens \sigma^{-1}(S(\two{h}), \three{h})\#_{\sigma} \one{h})\tens_B (\four{h}\triangleright d\tens b'\#_{\sigma} \five{h})\\
    =&t(\four{h}\triangleright d)(b\tens \sigma^{-1}(S(\two{h}), \three{h})\#_{\sigma} \one{h})\tens_B (1\tens b'\#_{\sigma} \five{h})\\
    =&b\ot \sigma^{-1}(S(\three{h}), \four{h})S(\two{h})\triangleright(\five{h}\triangleright d)\#_{\sigma}\one{h}\ot_{B}(1\ot b'\#_{\sigma} \six{h})\\
    =&(b\ot d \sigma^{-1}(S(\two{h}), \three{h})\#_{\sigma} \one{h})\tens_B 1\tens b'\#_{\sigma} \four{h})\\
    =&\one{X}t(d)\ot_{B}\two{X}.
\end{align*}
One can then show by direct computation that the coproduct is an algebra map and $\lambda^{-1}$ is the inverse of $\lambda$.
We omit details since this follows from a more general result Theorem~\ref{ESHopf} later.
\end{proof}

Next, we show that $B^{e}\#_{\sigma}H$ has an antipode in the sense of Lemma~\ref{lemS} at least in the associative type case.

\label{equ. inverse of antipode of cleft extension}

\begin{prop}\label{prop. antipode of cleft extension} Let $B\#_{\sigma}H$ be of associative type with an invertible antipode of $H$. Then
\begin{align*}
    \Ss (b\ot b'\#_{\sigma}h)&:=b'\ot \sigma^{-1}(S^{2}(\two{h}), S(\one{h}))S^{2}(\three{h})\triangleright b\#_{\sigma}S(\four{h})\\
    \Ss^{-1} (b\ot b'\#_{\sigma}h)&:=(\two{h}\triangleright b')\sigma(\three{h}, S(\four{h}))\ot  b\#_{\sigma}S^{-1}(\one{h}).
\end{align*}
is a left antipode for the Hopf algebroid $B^{e}\#_{\sigma}H$, where $S$ on the right denotes the antipode of $H$.
\end{prop}
\begin{proof}
Let $X=b\ot b'\#_{\sigma}h$ and define a linear map $\Ss: B^{e}\#_{\sigma}H\to B^{e}\#_{\sigma}H$ as stated  with the inverse given by the formula stated. Then
\begin{align*}
    \Ss^{-1} (\Ss (X))=&\Ss^{-1} (b'\ot \sigma^{-1}(S^{2}(\two{h}), S(\one{h}))S^{2}(\three{h})\triangleright b\#_{\sigma}S(\four{h}))\\
    =&\two{g}\la(\sigma^{-1}(S(\six{g}),  \seven{g})S(\five{g})\la b)\sigma(\three{g}, S(\four{g}))\ot b'\#_{\sigma}S^{-1}(\one{g})\\
    =&\two{g}\la((S(\eight{g})\la b))\three{g}\la(\sigma^{-1}(S(\six{g}),  \seven{g}))\sigma(\four{g}, S(\five{g}))\ot b'\#_{\sigma}S^{-1}(\one{g})\\
    =&b\ot b'\#_{\sigma}S^{-1}(g)
    =X.
\end{align*}
In the 2nd step, we let $g:=S(h)$. In the 3rd step we use Lemma~\ref{lem: associative type}. In the 4th step we use (4) of Lemma \ref{lemma. unital twist}. On the other side,
\begin{align*}
    \Ss (\Ss^{-1} (X))=&\Ss ((\two{h}\triangleright b')\sigma(\three{h}, S(\four{h}))\ot  b\#_{\sigma}S^{-1}(\one{h}))\\
    =&b\ot \sigma^{-1}(S(\three{h}), \four{h})S(\two{h})\la((\five{h}\la b')\sigma(\six{h}, S(\seven{h})))\#_{\sigma}\one{h}\\
    =&b\ot S(\four{h})\la((\five{h}\la b')\sigma(\six{h}, S(\seven{h})))\sigma^{-1}(S(\two{h}), \three{h})\#_{\sigma}\one{h}\\
    =&b\ot b' \sigma^{-1}(S(\three{h}), \four{h})S(\two{h})\la\sigma(\five{h}, S(\six{h})))\#_{\sigma}\one{h}\\
    =&b\ot b'\#_{\sigma} h.
\end{align*}
In the 3rd and 4th steps we use Lemma~\ref{lem: associative type}. In the 5th step we use (6) of Lemma \ref{lemma. unital twist}.

It remains to check that $S$ and its inverse satisfy the conditions of Lemma \ref{lemS}. By direct computation,
\begin{align*}
    \Ss (t(d)X)=b'(S(\five{h})\triangleright d)\ot \sigma^{-1}(S^{2}(\two{h}), S(\one{h}))S^{2}(\three{h})\triangleright b\#_{\sigma} S^{2}(\four{h})=\Ss (X)s(d).
\end{align*}
Similarly for $S^{-1}(s(b)X)=S(X)t(b)$.
For condition (2) of Lemma \ref{lemS}, we have
\begin{align*}
   &\kern-10pt \one{\Ss^{-1} (\two{X})}\ot_{B}\one{\Ss ^{-1}(\two{X})}\one{X}\\
    =&(\eight{h}\triangleright b')\sigma(\nine{h}, S(\ten{h}))\ot \sigma^{-1}(\six{h}, S^{-1}(\five{h}))\#_{\sigma}S^{-1}(\seven{h})\\
    &\quad \ot_{B}(1\ot 1\#_{\sigma} S^{-1}(\four{h}))(b\ot \sigma^{-1}(S(\two{h}), \three{h})\#_{\sigma} \one{h})\\
    =&(\seven{h}\triangleright b')\sigma(\eight{h}, S(\nine{h}))\ot \sigma^{-1}(\five{h}, S^{-1}(\four{h}))\#_{\sigma}S^{-1}(\six{h})\\
    &\quad \ot_{B} s((S^{-1}(\three{h})\triangleright b)\sigma(S^{-1}(\two{h}), \one{h}))\\
    =&(\eight{h}\triangleright b')\sigma(\nine{h}, S(\ten{h}))\ot \sigma^{-1}(\five{h}, S^{-1}(\four{h}))\six{h}\la ((S^{-1}(\three{h})\triangleright b)\sigma(S^{-1}(\two{h}), \one{h}))\\ &\quad  \#_{\sigma}S^{-1}(\seven{h}) \ot_{B} 1_{\cL}\\
    =&(\eight{h}\triangleright b')\sigma(\nine{h}, S(\ten{h}))\ot \four{h}\la ((S^{-1}(\three{h})\triangleright b)\sigma(S^{-1}(\two{h}), \one{h})) \sigma^{-1}(\six{h}, S^{-1}(\five{h}))\\ &\quad \#_{\sigma}S^{-1}(\seven{h}) \ot_{B} 1_{\cL}\\
    =&(\seven{h}\triangleright b')\sigma(\eight{h}, S(\nine{h}))\ot b (\three{h}\triangleright \sigma(S^{-1}(\two{h}), \one{h})) \sigma^{-1}(\five{h}, S^{-1}(\four{h}))\#_{\sigma}S^{-1}(\six{h})\ot_{B} 1_{\cL}\\
    =&(\seven{h}\triangleright b')\sigma(\eight{h}, S(\nine{h}))\ot b \sigma^{-1}(\four{h}, S^{-1}(\three{h}))(\five{h}\triangleright \sigma(S^{-1}(\two{h}), \one{h})) \#_{\sigma}S^{-1}(\six{h})\ot_{B} 1_{\cL}\\
    =&(\two{h}\triangleright b')\sigma(\three{h}, S(\four{h}))\ot  b\#_{\sigma}S^{-1}(\one{h})\ot_{B} 1_{\cL}\\
    =&\Ss^{-1} (X)\ot_{B} 1_{\cL},
\end{align*}
where the 4th and 6th steps use Lemma~\ref{lem: associative type}, the 7th step use (6) of Lemma \ref{lemma. unital twist}.
For condition (3),
\begin{align*}
    &\kern-15pt \one{\Ss (\one{X})}\two{X}\ot_{B}\two{\Ss (\one{X})}\\
    =&(\sigma^{-1}(S(\eight{h}), \nine{h})\ot \sigma^{-1}(S^{2}(\six{h}), S(\five{h}))\#_{\sigma}S(\seven{h}))(1\ot b'\#_{\sigma} \ten{h})\\
   &\quad \ot_{B}1\ot \sigma^{-1}(S^{2}(\two{h}), \one{h})S^{2}(\three{h})\triangleright b\#_{\sigma} S(\four{h})\\
    =&\sigma^{-1}(S(\ten{h}), \eleven{h})\sigma(S(\nine{h}), \twelve{h})\ot b' \big(S(\fifteen{h})\la \sigma^{-1}(S^{2}(\six{h}), S(\five{h}))\big)\sigma(S(\fourteen{h}), S^{2}(\seven{h}))\\ &\quad \#_{\sigma} S(\eight{h})\thirteen{h} \ot_{B}1\ot \sigma^{-1}(S^{2}(\two{h}), \one{h})S^{2}(\three{h})\triangleright b\#_{\sigma} S(\four{h})\\
    =&1\ot b' \big(S(\nine{h})\la \sigma^{-1}(S^{2}(\six{h}), S(\five{h}))\big)\sigma(S(\eight{h}), S^{2}(\seven{h}))\#_{\sigma} 1\\
   &\quad  \ot_{B}1\ot \sigma^{-1}(S^{2}(\two{h}), \one{h})S^{2}(\three{h})\triangleright b\#_{\sigma} S(\four{h})\\
    =&t(b')\ot_{B}1\ot \sigma^{-1}(S^{2}(\two{h}), \one{h})S^{2}(\three{h})\triangleright b\#_{\sigma} S(\four{h})\\
    =&1_{\cL}\ot_{B}b'\ot \sigma^{-1}(S^{2}(\two{h}), \one{h})S^{2}(\three{h})\triangleright b\#_{\sigma} S(\four{h})\\
    =&1_{\cL}\ot_{B}\Ss (X),
\end{align*}
where the 4th step uses Lemma~\ref{lem: associative type}. \end{proof}

\subsection{Cocycle Hopf algebroid in the associative type case}

We now use the above class of cocycle Hopf algebroids to illustrate our twisting theory.

\begin{lem}\label{lem. 2-cocycle on bialgebroid}
If $B\#_{\sigma} H$ is of associative type then the linear map $\tilde{\sigma}\in {}_{B}\Hom_{B}(B^{e}\# H\otimes_{B^e} B^{e}\# H)$  given by
\begin{align*}
    \tilde{\sigma}(b\ot b'\# h, c\ot c'\# g):=b(\one{h}\triangleright c)\sigma(\two{h}, \one{g})((\three{h}\two{g})\triangleright c')(\four{h}\triangleright b')
\end{align*}
for all $b, b', c, c'\in B$ and $h,g\in H$  is  an  invertible  normalised 2-cocycle.
\end{lem}
\begin{proof}
Since $B\#_{\sigma} H$ is of associative type, we know that $B$ is a left $H$-module algebra so that the smash product $B\#H$ is well defined.
Let $X=b\ot b'\# h$, $Y=c\ot c'\# g$ and $Z=d\ot d'\# f$ be three elements in $B^{e}\# H$.
First we show that this cocycle is well defined over the balanced tensor over $B^e\#H$.  On the one hand, we have
\begin{align*}
    \tilde{\sigma}(X\eta(d\otimes d'), Y)) &=\tilde{\sigma}((b\ot b'\# h)\eta(d\otimes d'), c\ot c'\# g)\\
     &=b(\one{h}\triangleright d)(\two{h}\triangleright c)\sigma(\three{h}, \one{g})((\four{h}\two{g})\triangleright c')(\five{h}\triangleright(d'b')),
\end{align*}
for all $d, d'\in B$. On the other hand,
\begin{align*}
     \tilde{\sigma}(X, \eta(d\otimes d')Y))&=\tilde{\sigma}((b\ot b'\# h), \eta(d\otimes d')(c\ot c'\# g))\\
     &=\tilde{\sigma}(b\ot b'\# h, dc\otimes c'(S(\two{g})\triangleright d')\#\one{g})\\
     &=b(\one{h}\triangleright(dc))\sigma(\two{h}, \one{g})((\three{h}\two{g})\triangleright(c'(S(\three{g})\triangleright d')))(\four{h}\triangleright b')\\
     &=b(\one{h}\triangleright d)(\two{h}\triangleright c)\sigma(\three{h}, \one{g})((\four{h}\two{g})\triangleright c')(\five{h}\triangleright(d'b')).
\end{align*}
The inverse $\tilde{\sigma}^{-1}$ is given by
\begin{align*}
     \tilde{\sigma}^{-1}(X, Y):=b(\one{h}\triangleright c)\sigma^{-1}(\two{h}, \one{g})((\three{h}\two{g})\triangleright c')(\four{h}\triangleright b').
\end{align*}
We can see that
\begin{align*}
 \tilde{\sigma}^{-1}\star  \tilde{\sigma}(X, Y)&=\tilde{\sigma}^{-1}\star  \tilde{\sigma}(b\ot b'\# h, c\ot c'\# g)\\
 &=b(\one{h}\triangleright c)\sigma^{-1}(\two{h}, \one{g})\sigma(\three{h}, \two{g})((\four{h}\three{g})\triangleright c')(\five{h}\triangleright b')\\
 &=\epsilon((b\ot b'\# h)(c\ot c'\# g))\\
 &=\tilde{\epsilon}(X, Y),
\end{align*}
where $\tilde{\epsilon}$ is the unit in the algebra ${}_{B}\Hom_{B}(B^{e}\# H\otimes_{B^e} B^{e}\# H, B)$. Similarly for  $\tilde{\sigma}\star  \tilde{\sigma}^{-1}=\tilde{\epsilon}$.

It is clear that $\tilde{\sigma}$ is left $B$-linear. We show that it is also right $B$-linear,
\begin{align*}
    \tilde{\sigma}(t(b'')X, Y)&=\tilde{\sigma}(t(b'')(b\ot b'\# h, c\ot c'\# g)\\
    &=\tilde{\sigma}(b\otimes b'(S(\two{h})\triangleright b'')\# \one{h}, c\ot c'\# g)\\
    &=b(\one{h}\triangleright c)\sigma(\two{h}, \one{g})((\three{h}\two{g})\triangleright c')(\four{h}\triangleright (b'(S(\five{h})\triangleright b'')))\\
    &=b(\one{h}\triangleright c)\sigma(\two{h}, \one{g})((\three{h}\two{g})\triangleright c')(\four{h}\triangleright b')b''\\
    &=\tilde{\sigma}(X, Y)b'',
\end{align*}
for all $b''\in B$. Now we show the cocycle condition of $\tilde{\sigma}$. On the one hand,
\begin{align*}
    &\tilde{\sigma}(X, s(\tilde{\sigma}(\one{Y}, \one{Z}))\two{Y}\two{Z})\\
    &=\tilde{\sigma}(b\ot b'\# h, c(\one{g}\triangleright d)\sigma(\two{g}, \one{f})\otimes d'(S(\three{f})\triangleright c')\# \three{g}\two{f})\\
      &=b\Big(\one{h}\triangleright \big(c(\one{g}\triangleright d)\sigma(\two{g}, \one{f})\big)\Big)\sigma(\two{h}, \three{g}\two{f})\big((\three{h}\four{g}\three{f})\triangleright (d'(S(\four{f})\triangleright c'))\big)(\four{h}\triangleright b').
\end{align*}
On the other hand,
\begin{align*}
    \tilde{\sigma}(s(\tilde{\sigma}&(\one{X}, \one{Y}))\two{X}\two{Y}, Z)\\
    &=\tilde{\sigma}(b(\one{h}\triangleright c)\sigma(\two{h}, \one{g})\otimes c' (S(\three{g})\triangleright b')\#\three{h}\two{g}, d\otimes d'\# f)\\
    &=b(\one{h}\triangleright c)\sigma(\two{h}, \one{g})((\three{h}\two{g})\triangleright d)\sigma(\four{h}\three{g}, \one{f})((\five{h}\four{g}\two{f})\triangleright d')\\ &\quad \big((\six{h}\five{g})\triangleright(c'(S(\six{g})\triangleright b'))\big).
\end{align*}
Comparing the results on both hand sides, it is sufficient to show
\begin{align*}
    \one{h}\triangleright((\one{g}\triangleright d)\sigma(\two{g}, \one{f}))\sigma(\two{h}, \three{g}\two{f})=\sigma(\one{h}, \one{g})((\two{h}\two{g})\triangleright d)\sigma(\three{h}\three{g}, f).
\end{align*}
This holds as
\begin{align*}
     \one{h}\triangleright((\one{g}\triangleright d)&\sigma(\two{g}, \one{f}))\sigma(\two{h}, \three{g}\two{f})\\
     &=(\one{h}\triangleright (\one{g}\triangleright d))(\two{h}\triangleright\sigma(\two{g}, \one{f}))\sigma(\three{h}, \three{g}\two{f})\\
     &=(\one{h}\triangleright (\one{g}\triangleright d))\sigma(\two{h}, \two{g})\sigma(\three{h}\three{g}, f)\\
     &=\sigma(\one{h}, \one{g})((\two{h}\two{g})\triangleright d)\sigma(\three{h}\three{g}, f).
\end{align*}
Finally, we have the normalisation condition
\begin{align*}
    &\tilde{\sigma}(X, 1)=b(\one{h}\triangleright b')=\epsilon(X),\quad \tilde{\sigma}(1, X)=b(\one{h}\triangleright b')=\epsilon(X).
    \end{align*}
By using the same methods, one can check $\tilde{\sigma}^{-1}$ satisfies the definition of normalised right 2-cocycle. This is another long calculation, but sufficiently similar that we omit the details. As a result, $\tilde{\sigma}$ is an invertible normalised 2-cocycle. \end{proof}

\begin{lem}\label{lem. isomorphic map between bialgebroid}
If  $B\#_{\sigma} H$ is of associative type then the map
\begin{align*}
 \phi: B^{e}\# H\to B^{e}\#_{\sigma} H,\quad   \phi( b\otimes b'\# h):=b\otimes b'\sigma^{-1}(S(\two{h}), \three{h})\#_{\sigma}\one{h}
\end{align*}
for all $b,b'\in B$ and $h\in H$  is an invertible $B$-coring map.
\end{lem}
\begin{proof}
First we check $\epsilon=\epsilon^{\sigma}\circ\phi$, where $\epsilon^{\sigma}$ is the counit of $B^{e}\#_{\sigma} H$.
Let $X=b\otimes b'\# h\in B^{e}\# H$, then
\begin{align*}
    \epsilon^{\sigma}(\phi(X))
    &=b(\one{h}\triangleright b')(\two{h}\triangleright \sigma^{-1}(S(\five{h}), \six{h}))\sigma(\three{h}, S(\four{h}))\\
    &=b(h\triangleright b')=\epsilon(b\otimes b'\# h),
\end{align*}
where in the 2nd step we use Proposition \ref{lemma. unital twist}.
%Here we always identify $B$ with its image in $B\#H$ and $B\#_{\sigma}H$ by $b\mapsto b\#1$ and $b\mapsto b\#_{\sigma}1$ respectively.
It is easy to see that $\phi$ is a left $B$-module map. That it is right $B$-linear is
\begin{align*}
    \phi(X.b'')=&\phi(b\otimes b' (S(\two{h})\triangleright b'')\#\one{h})\\
    =&b\otimes b' (S(\four{h})\triangleright b'')\sigma^{-1}(S(\two{h}), \three{h})\#_{\sigma} \one{h}\\
    =&b\otimes b' \sigma^{-1}(S(\three{h}), \four{h})(S(\two{h})\triangleright b'')\#_{\sigma}\one{h}\\
    =&\phi(X).b'',
\end{align*}
for all $b''\in B$, where the third step uses Lemma~\ref{lem: associative type}. Next, we have
\begin{align*}
    (\phi\otimes_{B}\phi)(\Delta(X))&=(\phi\otimes_{B}\phi)(b\otimes 1\#\one{h}\otimes_{B} 1\otimes b'\# \two{h})\\
    &=(b\otimes \sigma^{-1}(S(\two{h}), \three{h})\#_{\sigma}\one{h})\otimes_{B}(1\otimes b'\sigma^{-1}(S(\five{h}), \six{h})\#_{\sigma}\four{h})\\
    &=\Delta(b\otimes b'\sigma^{-1}(S(\two{h}), \three{h})\#_{\sigma}\one{h})\\
    &=\Delta(\phi(X)),
\end{align*}
 using the coproduct of $B^e\#_{\sigma}H$. It remains to check that $\phi$ is invertible. In fact, the inverse of $\phi$ is given by
\begin{align*}
    \phi^{-1}(b\otimes b'\#_{\sigma} h)=b\otimes b'(S(\two{h})\triangleright\sigma(\three{h}, S(\four{h})))\#\one{h}.
\end{align*}
Here
\[ \phi\phi^{-1}(b\tens b'\#_{\sigma} h)=b\tens b'\sigma^{-1}(S(\three{h}), \four{h}) S(\two{h})\la \sigma(\five{h},S(\six{h}))\#_{\sigma} \one{h}\]
and using Lemma~\ref{lemma. unital twist}(2), we have
\begin{align*}
\sigma^{-1}&(S(\two{h}),\three{h}) S(\one{h})\la \sigma(\four{h},S(\five{h}))\\
=&\sigma^{-1}(S(\four{h}), \five{h})\sigma(S(\three{h}), \six{h})\sigma(S(\two{h})\seven{h}, S(\ten{h}))\sigma^{-1}(S(\one{h}), \eight{h}S(\nine{h}))=\eps(h).
\end{align*}
In the other side, we similarly have
\[ \phi^{-1}\phi(b\tens b'\# h)=b\tens b'\sigma^{-1}(S(\five{h}),\six{h}) S(\two{h})\la \sigma(\three{h},S(\four{h}))\# \one{h}\]
and using Lemma~\ref{lemma. unital twist}(4), we have
\begin{align*}\sigma^{-1}&(S(\four{h}),\five{h}) S(\one{h})\la \sigma(\two{h},S(\three{h}))=&(S(\one{h}))\la\left((\two{h}\la\sigma^{-1}(S(\five{h}),\six{h}))\sigma(\three{h},S(\four{h}))\right)\\
=&S(\one{h})\la \eps(\two{h})1=\eps(h)
\end{align*}
where, in the associative type case, we use that $\la$ is a module algebra.  \end{proof}

\begin{thm}\label{theorem. 2-cocycle twist}
If  $B\#_{\sigma} H$ is a associative type, then there is an invertible normalised 2-cocycle $\tilde{\sigma}$ on $B^{e}\#H$, such that
$\phi: (B^{e}\#H)^{\tilde{\sigma}}\to B^{e}\#_{\sigma}H$ is an isomorphism of left $B$-bialgebroids, where $\tilde{\sigma}$ is given by Lemma~\ref{lem. 2-cocycle on bialgebroid} and $\phi$ is given by Lemma~\ref{lem. isomorphic map between bialgebroid}.
\end{thm}
\begin{proof}
Since $\phi$ is a coring map, so we only need to show that $\phi$ is an algebra map. Let $X=b\otimes b'\# h$, $Y=c\otimes c'\# g\in B^{e}\#H$. On the one hand,
\begin{align*}
    \phi&(X\cdot_{\tilde{\sigma}}Y)\\
    &=\phi\Big(\tilde{\sigma}(b\otimes 1\#\one{h}, c\otimes 1\#\one{g})\otimes \big((S(\three{g})S(\three{h}))\triangleright \tilde{\sigma}^{-1}(1\otimes b'\#\four{h}, 1\otimes c'\#\four{g})\big)\#\two{h}\two{g}\Big)\\
    &=\phi\Big(b(\one{h}\triangleright c)\sigma(\two{h}, \one{g})\otimes(S(\three{g})S(\four{h}))\triangleright\big(\sigma^{-1}(\five{h}, \four{g})((\six{h}\five{g})\triangleright c')(\seven{h}\triangleright b')) \big)\#\three{h}\two{g}\Big)\\
    &=b(\one{h}\triangleright c)\sigma(\two{h}, \one{g})\otimes(S(\five{g})S(\six{h}))\triangleright\big(\sigma^{-1}(\seven{h}, \six{g})((\eight{h}\seven{g})\triangleright c')(\nine{h}\triangleright b')) \big)\\
    &\quad \sigma^{-1}(S(\three{g})S(\four{h}), \five{h}\four{g}))\#_{\sigma}\three{h}\two{g}.
\end{align*}
On the other hand,
\begin{align*}
 \phi(X)\phi(Y)
    &=(b\otimes b'\sigma^{-1}(S(\two{h}), \three{h})\#_{\sigma}\one{h})(c\otimes c'\sigma^{-1}(S(\two{g}), \three{g})\#_{\sigma}\one{g})\\
    &=b(\one{h}\triangleright c)\sigma(\two{h}, \one{g})\otimes\\
    &\quad c'\sigma^{-1}(S(\five{g}), \six{g})\big(S(\four{g})\triangleright(b'\sigma^{-1}(S(\five{h}), \six{h}))\big)\sigma(S(\three{g}), S(\four{h}))\#_{\sigma}\three{h}\two{g},
\end{align*}
Comparing the results on both sides, it is enough to show that
\begin{align*}
    (S(\three{g})S(\three{h}))&\triangleright\big(\sigma^{-1}(\four{h}, \four{g})((\five{h}\five{g})\triangleright c')(\six{h}\triangleright b')) \big)\sigma^{-1}(S(\one{g})S(\one{h}), \two{h}\two{g}))\\
    =&c'\sigma^{-1}(S(\three{g}), \four{g})\big(S(\two{g})\triangleright(b'\sigma^{-1}(S(\two{h}), \three{h}))\big)\sigma(S(\one{g}), S(\one{h})).
\end{align*}
Indeed,
\begin{align*}
    (S(\three{g})&S(\three{h}))\triangleright\big(\sigma^{-1}(\four{h}, \four{g})((\five{h}\five{g})\triangleright c')(\six{h}\triangleright b') \big)\sigma^{-1}(S(\one{g})S(\one{h}), \two{h}\two{g}))\\
    =&(S(\three{g})S(\three{h}))\triangleright\big(((\four{h}\four{g})\triangleright c')\sigma^{-1}(\five{h}, \five{g})(\six{h}\triangleright b') \big)\sigma^{-1}(S(\one{g})S(\one{h}), \two{h}\two{g}))\\
    =&c'(S(\three{g})S(\three{h}))\triangleright\big(\sigma^{-1}(\four{h}, \four{g})(\five{h}\triangleright b') \big)\sigma^{-1}(S(\one{g})S(\one{h}), \two{h}\two{g}))\\
    =&c'\sigma^{-1}(S(\two{g})S(\two{h}), \three{h}\three{g}))(S(\one{g})S(\one{h}))\triangleright\big(\sigma^{-1}(\four{h}, \four{g})(\five{h}\triangleright b') \big)\\
    =&c'\sigma^{-1}(S(\four{g}),\five{g})(S(\three{g})\triangleright \sigma^{-1}(S(\three{h}), \four{h}\six{g}))\\
    &\quad \sigma(S(\two{g}), S(\two{h}))(S(\one{g})S(\one{h}))\triangleright\big(\sigma^{-1}(\five{h}, \seven{g})(\six{h}\triangleright b') \big)\\
    =&c'\sigma^{-1}(S(\four{g}),\five{g})(S(\three{g})\triangleright \sigma^{-1}(S(\three{h}), \four{h}\six{g}))\\
    &(S(\two{g})S(\two{h}))\triangleright\big(\sigma^{-1}(\five{h}, \seven{g})(\six{h}\triangleright b') \big)\sigma(S(\one{g}), S(\one{h}))\\
    =&c'\sigma^{-1}(S(\three{g}),\four{g})S(\two{g})\triangleright \Big(\sigma^{-1}(S(\three{h}), \four{h}\five{g})S(\two{h})\triangleright\big(\sigma^{-1}(\five{h}, \six{g})(\six{h}\triangleright b') \big)\Big)\\
    &\quad \sigma(S(\one{g}), S(\one{h}))\\
    =&c'\sigma^{-1}(S(\three{g}),\four{g})S(\two{g})\triangleright \Big(\sigma^{-1}(S(\six{h}), \seven{h}\five{g})\sigma(S(\five{h}), \eight{h}\six{g})\sigma^{-1}(S(\four{h}) \nine{h},\seven{g})\\
    &\quad \sigma^{-1}(S(\three{h}),\ten{h})(S(\two{h})\eleven{h})\triangleright b')\Big)\sigma(S(\one{g}), S(\one{h}))\\
    =&c'\sigma^{-1}(S(\three{g}),\four{g})S(\two{g})\triangleright \Big(\sigma^{-1}(S(\three{h}),\four{h})S(\two{h})\triangleright(\five{h}\triangleright b')\Big)\sigma(S(\one{g}), S(\one{h}))\\
    =&c'\sigma^{-1}(S(\three{g}),\four{g})S(\two{g})\triangleright \Big(S(\four{h})\triangleright(\five{h}\triangleright b')\sigma^{-1}(S(\two{h}),\three{h})\Big)\sigma(S(\one{g}), S(\one{h}))\\
    =&c'\sigma^{-1}(S(\three{g}),\four{g})S(\two{g})\triangleright \Big(b'\sigma^{-1}(S(\two{h}),\three{h})\Big)\sigma(S(\one{g}), S(\one{h})),
\end{align*}
where the 1st, 5th steps use the property of associative type. The 3rd, 9th steps use Lemma~\ref{lem: associative type}. The 4th step uses Proposition \ref{lemma. unital twist}. Hence, $\phi(X\cdot_{\tilde{\sigma}}Y)=\phi(X)\phi(Y)$.
\end{proof}

The simplest case of a Galois object (see Section~\ref{sec5}) was already shown in \cite{schau}.  By Proposition~\ref{prop cocy Hopf algebroid} and Theorem \ref{theorem. 2-cocycle twist}, we have the following corollary

\begin{cor}\label{twistS}
If  $B\#_{\sigma} H$ is of associative type, then $(B^{e}\# H)^{\tilde{\sigma}}$ is a left Hopf algebroid with left antipode and its inverse given by
\begin{align*}
    S^{\tilde{\sigma}}(X)&:=\phi^{-1}\circ \Ss \circ \phi(X)
    =b'\sigma^{-1}(S(\three{h}), \four{h})\ot S^{2}(\one{h})\triangleright b\#S(\two{h})\\
    (S^{\tilde{\sigma}})^{-1}(X)&:=\phi^{-1}\circ \Ss {}^{-1}\circ \phi(X)
    =\five{h}\triangleright b'\ot b\three{h}\triangleright\sigma(S^{-1}(\two{h}), \one{h})\#S^{-1}(\four{h})
\end{align*}
as corresponding via $\phi$ to $S$ on $B^e\#_\sigma H$ in Proposition~\ref{prop. antipode of cleft extension}. \end{cor}
Moreover, computing $(\lambda^{\tilde\sigma})^{-1}(X\tens_B 1)$ from Theorem~\ref{thm. 2 cocycle twist}, we have
\begin{align*}
   X_{+}^{\tilde{\sigma}}\ot_{B^{op}} X_{-}^{\tilde{\sigma}}=&t(\tilde{\sigma}(\two{S^{-1}(\four{S(\one{X})})}, \three{S(\one{X})})) \one{S^{-1}(\four{S(\one{X})})}\\ & \quad
   \ot s((\tilde{\sigma})^{-1}(\one{S(\one{X})}, \two{X}))\two{S(\one{X})},
\end{align*}
where $S$ just above is the untwisted antipode.  Using this, we compute
\begin{align}
  t(\eps(S^{\tilde{\sigma}}(X_{+}^{\tilde{\sigma}})))X_{-}^{\tilde{\sigma}}= S^{\tilde{\sigma}}(X),
\end{align}
where we use Lemma~\ref{lemS} to recognise $S^{\tilde\sigma}$ in the above corollary. This suggests that there should be a general theory of twistings of antipodes with $\eps_R=\eps\circ S^{\tilde\sigma}$ constructed independently by twisting and in the spirit of the left-right antipode theory in \cite{Boehm}. This will be looked at elsewhere.

\begin{exa}\label{exa. inner action} Following on from Example \ref{ex inner}, we start with a smash product with inner action given by $u$ and no cocycle, and we saw that this is cohomologous to a trivial action and $\sigma^u$. By Lemma~\ref{lem. 2-cocycle on bialgebroid}, the corresponding 2-cocycle (which cotwists the left Hopf algebroid $B^e\# H$) is \begin{align*}
    \tilde{\sigma^{u}}(b\ot b'\#_{\sigma^{u}}h, c\ot c'\#_{\sigma^{u}}g)&=bc\sigma^{u}(h, g)c'b'=bcu^{-1}(\one{g})u^{-1}(\one{h})u(\two{h}\two{g})c'b'\\
    (\tilde{\sigma^{u}})^{-1}(b\ot b'\#_{\sigma^{u}}h, c\ot c'\#_{\sigma^{u}}g)&=bc(\sigma^{u})^{-1}(h, g)c'b'=bcu^{-1}(\one{h}\one{g})u(\two{h})h(\two{g})c'b'
\end{align*}
and the resulting Hopf algebroid $(B^e\# H)^{\tilde{\sigma^u}}$ has, by  Corollary \ref{twistS},
\begin{align*}
   S^{\tilde{\sigma^{u}}}(b\ot b'\#h)&=b'u(S(\two{h}))u(\three{h})\ot b\# S(\one{h})\\
   (S^{\tilde{\sigma^{u}}})^{-1}(b\ot b'\#h)&=b'\ot u^{-1}(\one{h})u^{-1}(S^{-1}(\two{h}))b\# S^{-1}(\three{h}).
\end{align*}
Moreover, this Hopf algebroid is isomorphic to $B^e\#_{\sigma^u}H$, which by Proposition \ref{prop. antipode of cleft extension}, has antipode  and its inverse
\begin{align*}
    S(b\ot b'\#_{\sigma^{u}}h)=&b'\ot (\sigma^{u})^{-1}(S^{2}(\two{h}), S(\one{h}))b\#_{\sigma^{u}}S(\three{h})\\
    =&b'\ot u(S^{2}(\two{h}))u(S(\one{h}))b\#_{\sigma^{u}}S(\three{h})\\
    S^{-1}(b\ot b'\#_{\sigma^{u}}h)=&b'\sigma^{u}(\two{h}, S(\three{h}))\ot b\#_{\sigma^{u}}S^{-1}(\one{h})\\
    =&b' u^{-1}(S(\three{h}))u^{-1}(\two{h})\ot b\#_{\sigma^{u}}S^{-1}(\one{h}).
\end{align*}
If $u$ is an algebra map then $\sigma^u(h,g)=\eps(h)\eps(g)$ and the antipode is just the interchange of $b$ and $b'$ together with $S$ on $H$.
\end{exa}

\section{Results for general Hopf Galois extensions}\label{sec4}

Here, we extend some of our results on cocycle extensions to a more general context of Hopf Galois extensions.

\subsection{Preliminaries on Hopf Galois extensions}\label{sec:chge}

A \textup{Hopf-Galois extension} or quantum principal bundle with universal calculus means an $H$-comodule algebra $P$ with coinvariant subalgebra $B:=\big\{b\in P ~|~ \Delta_R (b) = b \ot 1_H \big\} \subseteq P$ such that the \textup{canonical map}
\begin{align*}
    \can: P\otimes_{B}P\to P\otimes H, \quad p\otimes_{B}q\mapsto p\zero{q}\otimes \one{q}
\end{align*}
is bijective, where $\otimes_{B}$ is the balanced tensor product by $B\subseteq P$ as a subalgebra. We also require for convenience that $P$ is a faithfully flat left $B$-module. In fact the inverse of $\can$ is determined by its restriction, the \textit{translation map},
\begin{align*}
    \tau:=\can^{-1}|_{1\otimes H}:  H\to P\otimes_{B}P,\qquad h\mapsto \tuno{h}\otimes_{B}\tdue{h}
\end{align*}
and it is known, e.g. \cite[Prop. 3.6]{brz-tr}\cite[Lemma 34.4]{BW} that it obeys
\begin{align}\label{equ. translation map 1}
  \tuno{h} \ot_B \zero{\tdue{h}} \ot \one{\tdue{h}} &= \,\tuno{\one{h}} \ot_B \tdue{\one{h}} \ot
\two{h},\\
\label{equ. translation map 2}
~~ \tuno{\two{h}}  \ot_B \tdue{\two{h}} \ot S(\one{h}) &= \zero{\tuno{h}} \ot_B {\tdue{h}}  \ot \one{\tuno{h}},\\
\label{equ. translation map 3}
\tuno{h}\zero{\tdue{h}}\ot \one{\tdue{h}} &= 1_{P} \ot h,\\
\label{equ. translation map 4}
    \zero{p}\tuno{\one{p}}\ot_{B}\tdue{\one{p}} &= 1_{P} \ot_{B}p
\end{align}
for all $h\in H$ and $p\in P$.

Finally, a Hopf-Galois extension is {\em cleft} if there is a {\em cleaving map} $j:H\to P$ which is convolution-invertible and $H$-equivariant.  This is known cf \cite[Theorem 8.2.4]{mont} to be equivalent to requiring that $P \simeq B\otimes H$ as left $B$-modules and right $H$-comodules (i.e. that it has the `normal basis property'). This means that a cleft Hopf Galois extension is equivalent to a cocycle cross
product. The converse is also well known but it is useful to recap the proof.
\begin{lem}
$B\#_{\sigma} H$ for a cocycle $(\sigma, \triangleright)$ is a cleft Hopf-Galois extension. Hence the two notions
are equivalent.
\end{lem}
\begin{proof} We have already seen how the data of an extension with the normal basis property implies a cocycle cross product. Conversely,  if we are given a cocycle cross product $B\#_{\sigma}H$ by data $(\la,\sigma)$ then we can view it as a cleft Hopf-Galois extension with $\Delta_R=\id\tens\Delta$ and $B\tens 1=(B\#_{\sigma}H)^{co(H)}$, $j(h)=1\tens h$ as previously discussed. This is clear but we explain it explicitly. Indeed, the inverse to the canonical
map is provided by\cite{mont}
\begin{align}\label{equ. inverse of canonical map of cleft extension}
    \can^{-1}(b\#_{\sigma}g\otimes h)=(b\#_{\sigma}g)(\sigma^{-1}(S(\two{h}), \three{h})\#_{\sigma} S(\one{h}))\otimes_{B}1\#_{\sigma}\four{h},
\end{align}
for all $b\in B$ and $g, h\in H$. One can check that
\begin{align*}
    \can\big((b\#_{\sigma}g)&(\sigma^{-1}(S(\two{h}), \three{h})\#_{\sigma}S(\one{h}))\otimes_{B}1\#_{\sigma}\four{h}\big)\\
    &=(b\#_{\sigma}g)\big(\sigma^{-1}(S(\three{h}), \four{h})\sigma(S(\two{h}), \five{h})\#_{\sigma}S(\one{h})\six{h}\big)\otimes \seven{h}=b\#_{\sigma}g\otimes h.\\
\can^{-1}(\can(b\#_{\sigma}&g\otimes_{B}b'\#_{\sigma}h))\\
&=(b\#_{\sigma}g)(b'\#_{\sigma}\one{h})(\sigma^{-1}(S(\three{h}), \four{h})\#_{\sigma}S(\two{h}))\otimes_{B} 1\#_{\sigma}\five{h}\\
&=(b\#_{\sigma}g)\big(b'(\one{h}\triangleright \sigma^{-1}(S(\six{h}), \seven{h}))\sigma(\two{h}, S(\five{h}))\#_{\sigma}\three{h}S(\four{h})\big)\otimes_{B}1\#_{\sigma}\eight{h}\\
&=(b\#_{\sigma}g)(b'\#_{\sigma}1)\otimes_{B}1\#_{\sigma}h=b\#_{\sigma}g\otimes_{B}b'\#_{\sigma}h,
\end{align*}
for all $b, b'\in B$ and $g, h\in H$, where the third step uses Proposition \ref{lemma. unital twist}.
\end{proof}

Given the lemma, our various results in Section~\ref{sec2} translate to the point of view of cleft Hopf-Galois extensions as follows, with the more elementary proofs left to the reader. First of all, our notion of associative type from the point of view of cleft Hopf-Galois extensions is equivalent to
\begin{align}\label{HopfGal assoc}
    j^{-1}(\one{g})j^{-1}(\one{h})j(\two{h}\two{g})\in Z(B).
\end{align}

Next,  `gauge transform' from $(\la,\sigma)$ to $(\la^u,\sigma^u)$ by convolution-invertible $u:H\to B$ is equivalent  from the point of view of cleft Hopf-Galois extensions to a change of the cleaving map $j$ to $j^u=u^{-1}*j$, and preserves associative type as in Proposition~\ref{prop. Has} if and only if
\begin{equation}\label{HopfGal u assoc type} j^{-1}(g\o) u(g\t)j^{-1}(h\o)u(h\t) u^{-1}(h\th g\th)j(h\fo g\fo)\in Z(B).\end{equation}

Lastly, it was recently observed \cite{ppca}  that if $\chi: H\otimes H\to k$ is an invertible normalised Drinfeld cotwist on  $H$ then a Hopf-Galois extension $B=P^{coH}\subseteq P$ cotwists to another one,  $B=P_{\chi}^{coH^{\chi}}\subseteq P_{\chi}$ with product (\ref{eqn. P cotwist}). We  make this a little more explicit:

\begin{lem}\label{lem. deformed translation map} If $P,H$ is a Hopf-Galois extension with translation map as above and $\chi$ a Drinfeld cotwist then $P_\chi,H^\chi$ has translation map
\begin{align*}
   \tau_{\chi}(h):=\tuno{\three{h}}\otimes_{B}\tdue{\three{h}}\chi(\one{h}, S(\two{h})),
\end{align*}
for all $h\in H$.
\end{lem}
\begin{proof} We give an elementary direct proof. On the one hand, \begin{align*}
    (\can_{\chi}\circ \tau_{\chi})(h)&=\zero{\tuno{\three{h}}}\zero{\tdue{\three{h}}}\otimes \chi^{-1}(\one{\tuno{\three{h}}}, \one{\one{\tdue{\three{h}}}})\two{\one{\tdue{\three{h}}}}\chi(\one{h}, S(\two{h}))\\
    &=\tuno{\four{h}}\zero{\tdue{\four{h}}}\otimes \chi^{-1}(S(\three{h}), \one{\one{\tdue{\four{h}}}})\two{\one{\tdue{\four{h}}}}\chi(\one{h}, S(\two{h}))\\
    &=1\otimes \chi^{-1}(S(\three{h}),\four{h})\five{h}\chi(\one{h}, S(\two{h}))=1\otimes h,
\end{align*}
for all $h\in H$, where the 2nd step uses (\ref{equ. translation map 2}), the 3rd step uses (\ref{equ. translation map 3}), and the last step uses the fact that $\chi$ is a cotwist on the Hopf algebra $H$. (This is sufficient since $\can_\chi$ is a left $P_\chi$-module map.) On the other side,
\begin{align*}
    (\can_{\chi}^{-1}\circ \can_{\chi})(p\otimes_{B} q)&=p\cdot_{\chi} \zero{q}\cdot_{\chi}\tuno{\three{q}}\otimes_{B}\tdue{\three{q}}\chi(\one{q}, S(\two{q}))\\
    &=p\cdot_{\chi}(\zero{q}\zero{\tuno{\four{q}}})\otimes_{B}\tdue{\four{q}}\chi^{-1}(\one{q}, \one{\tuno{\four{q}}})\chi(\two{q}, S(\three{q}))\\
    &=p\cdot_{\chi}(\zero{q}\tuno{\five{q}})\otimes_{B}\tdue{\five{q}}\chi^{-1}(\one{q}, S(\four{q}))\chi(\two{q}, S(\three{q}))\\
    &=p\cdot_{\chi}(\zero{q}\tuno{\one{q}})\otimes_{B}\tdue{\one{q}}=p\otimes_{B}q
\end{align*}
for all $p, q\in P_{\chi}$, where the 3rd step uses (\ref{equ. translation map 2})  and the last step uses (\ref{equ. translation map 4}). Now, in the cleft case, as the isomorphism $P\cong B\otimes H$ is left $B$-linear (since the coaction is the same on the underlying vector space) and right $H^\chi$-colinear (since $b\cdot_{\chi} p=bp$ as $b$ is coinvariant), the twisting reduces to Proposition~\ref{prop. twist cleft Galois object}.\end{proof}

In the cleft case, this agrees with the twisting of the cocycle extension in Proposition~\ref{prop. twist cleft Galois object}. From the point of view of Hopf-Galois extensions, the condition that  twisting by $\chi$ preserves associative type as in Proposition~\ref{prop. cocy as twist}  is
\begin{align}\label{HopfGal assoc chi}
   j^{-1}(h\o g\o) \chi(\two{h}, \two{g})j(\three{h}\three{g}) \in Z(B)\end{align}
   on using the form of the action $\la$ as given by $j$.

\subsection{Ehresmann-Schauenburg Hopf algebroid}

 The classical construction here is the Ehresmann groupoid $\mathcal G$ associated to a classical principal bundle $\pi:X\to M$ where $M$ is classical manifold, with structure group $G$. Then ${\mathcal G}=(X\times X)/G$ where we identify $(x,y)\sim (x.u,y.u)$ for all $u\in G$ and $x,y\in X$. The base is $M=X/G$ similarly identified. The composition is $(x,y)\circ  (y, z)=(x,z)$ and the source and target maps are represented by $s(x,y)=x$ and $t(x,y)=y$, which one can check are all well defined on the quotient spaces (so $s,t$ send the class of $(x,y)$  to $\pi(x)$ and $\pi(y)$ respectively).
\begin{defi}\label{def:ec}\cite[\S 34.13]{BW}.
Let $B=P^{co H}\subseteq P$ be a Hopf Galois extension such that $P$ is a faithfully flat left $B$-module. The \textup{Ehresmann Schauenburg bialgebroid} is
\begin{equation*}
\CB:=\{p\otimes q\in P\otimes P\quad|\quad \zero{p}\otimes \tau(\one{p})q=p\otimes q\otimes_B 1\}\subset P\tens P
\end{equation*}
 with $B$-bimodule inherited from $P$ and $B$-coring coproduct and counit
 \[
\Delta(p\otimes q)=\zero{p}\otimes \tau(\one{p})\otimes q,\quad  \epsilon(p\otimes q)=pq.
\]
Moreover, $\CB$ is a $B^e$-ring with the product and unit
\[ (p\otimes q)\bullet_{\mathcal{C}}(r\otimes u)=pr\otimes uq,\quad \eta(b\tens c)=b\tens c
\]
for all $p\otimes q$, $r\otimes u \in \mathcal{B}$ and $b\tens c\in B^e$. Here $s(b)=b\otimes 1$ and $t(b)=1\otimes b$.
\end{defi}

It was shown in \cite{BW,HL20} that the stated condition on $p\tens q$ (sum of terms understood) to be in $\mathcal{B}$ is equivalent to being in the space $(P\ot P)^{coH}$ defined by
\beq
(P\ot P)^{coH} := \{p\ot  q\in P\ot  P \, \quad|\quad \zero{p}\ot  \zero{q}\ot  \one{p}\one{q}=p\ot  q\ot  1_H \}. \label{ec2}
\eeq
We will prefer to use this as the characterisation of $\mathcal{B}$ in line with the classical description. Then the fact that it is true classically suggests that $\CB$ should  generally be a left Hopf algebroid. We prove this now.

\begin{thm}\label{ESHopf}
If $B\subseteq P$ is a Hopf Galois extension with Hopf algebra $H$ then $\CB$ is a Hopf algebroid which we denote $\CL(P,H)$.
\end{thm}
\begin{proof}
Let $X=p\ot q$, $Y=p'\ot q'$. We define $\lambda^{-1}: \CL(P, H)\ot_{B}\CL(P, H)\to \CL(P, H)\ot_{B^{op}}\CL(P, H)$ by
\begin{align*}
    \lambda^{-1}(X\ot_{B}Y):=p\ot\tdue{\one{q}}\ot_{B^{op}}(\zero{q}\ot\tuno{\one{q}})Y,
\end{align*}
which is well defined  since  firstly,
\begin{align*}
    \lambda^{-1}(t(b)X\ot_{B}Y)=&p\ot\tdue{\one{q}}\ot_{B^{op}}(\zero{q}b\ot\tuno{\one{q}})Y=\lambda^{-1}(X\ot_{B}s(b)Y).
\end{align*}
Also, for the formula itself to make sense given the output of $\tau$ is in $\tens_B$,
\begin{align*}
    p\ot b\tdue{\one{q}}\ot_{B^{op}}(\zero{q}\ot\tuno{\one{q}})Y=&(p\ot \tdue{\one{q}})t(b)\ot_{B^{op}}(\zero{q}\ot\tuno{\one{q}})Y\\
    =&p\ot \tdue{\one{q}}\ot_{B^{op}}t(b)(\zero{q}\ot\tuno{\one{q}})Y\\
    =&p\ot \tdue{\one{q}}\ot_{B^{op}}(\zero{q}\ot\tuno{\one{q}}b)Y
\end{align*}
Moreover, $(p\ot\tdue{\one{q}})\ot_{B^{op}}(\zero{q}\ot\tuno{\one{q}})\in \CL(P, H)\ot_{B^{op}}\CL(P, H)$. For example, the coaction on $P\tens P$ applied to the left factor,
\begin{align*}
    \Delta_{R}(p\ot \tdue{\one{q}})&\ot_{B^{op}}(\zero{q}\ot\tuno{\one{q}})\\
    =&(\zero{p}\ot \zero{\tdue{\one{q}}})\ot \one{p}\one{\tdue{\one{q}}}\ot_{B^{op}}(\zero{q}\ot\tuno{\one{q}})\\
    =&(\zero{p}\ot \tdue{\one{q}})\ot \one{p}\two{q}\ot_{B^{op}}(\zero{q}\ot\tuno{\one{q}})\\
    =&(p\ot \tdue{\one{q}})\ot 1\ot_{B^{op}}(\zero{q}\ot\tuno{\one{q}}).
\end{align*}
Similarly for the other factor.

It remains to check that  $\lambda^{-1}$ is indeed the inverse of $\lambda$. On the one hand,
\begin{align*}
    \lambda(\lambda^{-1}(X\ot_{B}Y))=&\lambda(p\ot\tdue{\one{q}}\ot_{B^{op}}(\zero{q}\ot\tuno{\one{q}})Y)\\
    =&\zero{a}\ot\tuno{\zero{p}}\ot_{B}(\tdue{\one{p}}\ot \tdue{\one{q}})(\zero{q}\ot\tuno{\one{q}})Y\\
    =&\zero{p}\ot\tuno{\zero{p}}\ot_{B}(\tdue{\one{p}}q\ot 1)Y\\
    =&X\ot_{B}Y.
\end{align*}
On the other hand,
\begin{align*}
    \lambda^{-1}(\lambda(X\ot_{B^{op}}Y))=&\lambda^{-1}(\zero{p}\ot\tuno{\one{p}}\ot_{B}(\tdue{\one{p}}\ot q)Y)\\
    =&\zero{p}\ot\tdue{\one{\tuno{\one{p}}}}\ot_{B^{op}}(\zero{\tuno{\one{p}}}\ot \tuno{\one{\tuno{\one{p}}}})(\tdue{\one{p}}\ot q)Y\\
    =&\zero{p}\ot \tdue{S(\one{p})}\ot_{B^{op}}(\tuno{\two{p}}\ot\tuno{S(\one{p})})(\tdue{\two{p}}\ot q)Y\\
    =&\zero{p}\ot \tdue{S(\one{p})}\ot_{B^{op}}(1\ot q\tuno{S(\one{p})})Y\\=&\zero{p}\ot q\tuno{S(\one{p})}\tdue{S(\one{p})}\ot_{B^{op}}Y\\
    =&X\ot_{B^{op}}Y.
\end{align*}

\end{proof}

This generalises \cite{schau1} for the case of a Galois object (see Section~\ref{sec5}). We also connect to Section~\ref{sec3} as follows.

\begin{lem}\label{EScleftbialg} In the cleft case, we have $\CL(B\#_{\sigma}H, H)\cong B^e\#_\sigma H$ as left Hopf algebroids.\end{lem}
\begin{proof} Using (\ref{ec2}), invariant elements of $P\tens P$ are of the form
\begin{align*}
    \CL(B\#_{\sigma} H, H)=\{b\#_{\sigma}\one{h}\otimes b'\#_{\sigma} S(\two{h})\quad|\quad  b, b'\in B, h\in H\}.
\end{align*}

We can also see the coproduct is given by
\begin{align*}
    \Delta(b\#_{\sigma}\one{h}\otimes b'\#_{\sigma} S(\two{h}))=b\#_{\sigma}\one{h}\otimes \sigma^{-1}(S(\three{h}), \four{h})\#_{\sigma} S(\two{h})\otimes_{B} 1\#_{\sigma}\five{h}\otimes b'\#_{\sigma} S(\six{h}),
\end{align*}
since from (\ref{equ. inverse of canonical map of cleft extension}) we know the translation map $\tau: H\to B\#_\sigma H\otimes_{B}B\#_\sigma H$ is
\begin{align}\label{taucocy}
    \tau(h)=\sigma^{-1}(S(\two{h}), \three{h})\#_{\sigma}S(\one{h})\otimes_{B}1\#_{\sigma}\four{h},
\end{align}
for all $h\in H$. Next,
\[  H\to H\tens H,\quad  h\mapsto \one{h}\tens S(\two{h})\]
is a vector space inclusion with image exactly the elements of the required form.  So we have an induced isomorphism of vector spaces
\[\Theta: B^e\#_\sigma H \to  \CL(B\#_{\sigma}H, H),\quad \Theta(b\ot b'\#_{\sigma} h):=b\#_{\sigma}\one{h}\otimes b'\#_{\sigma} S(\two{h}).\]
It remains to show that this preserves the bialgebroid structures.

First, it is not hard to see this is a $B$-bimodule map,
\begin{align*}
   \Theta(c. (b\ot b'\#_{\sigma} h).c')=&\Theta(cb\ot b' (S(\two{h})\triangleright c')\#_{\sigma} \one{h}) \\
   =&cb\#_{\sigma}\one{h}\otimes b'(S(\three{h})\triangleright c')\#_{\sigma} S(\two{h}).
\end{align*}
We can see also $\Theta$ is a coring map,
\begin{align*}
    (\Theta\ot_{B}\Theta)(\Delta(b\ot b'\#_{\sigma} h))=&b\#_{\sigma} \one{h}\ot \sigma^{-1}(S(\three{h}), \four{h})\#_{\sigma}S(\two{h})\ot_{B}1\ot_{\sigma}\five{h}\ot b'\#_{\sigma}S(\six{h})\\
    =&\Delta(b\#_{\sigma}\one{h}\otimes b'\#_{\sigma} S(\two{h})).
\end{align*}
Moreover, $\eps(\Theta(X))=\eps(X)$ since
\begin{align*}
    \eps(b\#_{\sigma}\one{h}\otimes b'\#_{\sigma}S(\two{h}))=b(\one{h}\triangleright b')\sigma(\two{h}, S(\three{h}))=\eps(b\ot b'\#_{\sigma} h).
\end{align*}
Finally, we show that $\Theta(XY)=\Theta(X)\Theta(Y)$ for all $X, Y\in B^{e}\#_{\sigma}H$. Indeed,
\begin{align*}
    \Theta((b\tens & b'\#_{\sigma} h).(c\tens c'\#_{\sigma} g))\\
    =&\Theta(b(\one{h}\triangleright c)\sigma(\two{h}, \one{g})\ot c' (S(\four{g})\triangleright  b')\sigma(S(\three{g}), S(\four{h}))\#_{\sigma} \three{h}\two{g})\\
    =&(b\#_{\sigma}\one{h}\otimes b'\#_{\sigma} S(\two{h}))(c\#_{\sigma}\one{g}\otimes c'\#_{\sigma} S(\two{g}))
\end{align*}
\end{proof}

We next want to know if $\CL(P,H)$ has an antipode. In view of the lemma, we already know this if $\sigma$ is of associative type. Classically, we know that the inverse of a gauge groupoid is just the flip of the ordered pair. However, when we flip a element $p\ot q\in \CL(P, H)$, the result $q\ot p$ doesn't belong to $\CL(P, H)$, since $H$ is not commutative. Rather, we need some kind of braiding to generalise the flip and we have a natural candidate if $H$ is coquasitriangular.

We recall that a Hopf algebra $H$ is a coquasitrangular if there is a convolution invertible linear map $\CR: H\ot H\to k$ which satisfies
    \begin{align*}
       & \one{g}\one{h}\CR(\two{h}\ot \two{g})=\CR(\one{h}\ot\one{g})\two{h}\two{g}\\
        \CR(h\ot gf)=&\CR(\one{h}\ot g)\CR(\two{h}\ot f),\quad \CR(hg\ot f)=\CR(h\ot \one{f})\CR(g\ot \two{f})
    \end{align*}
 for all $h, g, f\in H$. It is known\cite[Prop 2.2.4]{Ma:book} that the antipode of a coquasitrangular Hopf algebra is invertible. Now, since $P$ is a right $H$-comodule, we know\cite{Ma:book} that $\CR$ induces a braiding
 \[ \Psi(p\tens q)=\zero{q}\tens \zero{p} \CR(p\o\tens q\o)\]
 which is an $H$-comodule map and invertible, with inverse
 \[ \Psi^{-1}(p\tens q)=\zero{q}\ot \zero{p} \CR^{-1}(\one{q}\ot\one{p}).\]
We use this for the notion of `flip' needed for the antipode.

 \begin{prop}\label{prop. antipode for coquasi} If $H$ is coquasitriangular then $\CL(P, H)$ has a left antipode.
\end{prop}
\begin{proof} Since $\Psi$ is a morphism in the category of $H$-comodules, it and its inverse restrict to the invariant subspace of $P\tens P$, hence we can set $S=\Psi^{-1}$. We know that this is invertible and now we show that the remaining conditions of Lemma~\ref{lemS} apply.

First, we can see $S(t(b)X)=S(X)s(b)$ and $S^{-1}(s(b)X)=S^{-1}(X)t(b)$.
For condition (2) of Lemma \ref{lemS}, let $X=p\ot q$. Then
\begin{align*}
    \one{S^{-1}(\two{X})}&\ot_{B}\ot_{B}\one{S^{-1}(\two{X})}\one{X}\\
        =&\one{(\zero{q}\ot \CR(\one{\tdue{\one{p}}}\ot \one{q})\zero{\tdue{\one{p}}})}\\
        &\quad \ot_{B}\two{(\zero{q}\ot \CR(\one{\tdue{\one{p}}}\ot \one{q})\zero{\tdue{\one{p}}})}(\zero{p}\ot \tuno{\one{p}})\\
    =&(\CR(\two{p}\ot S(\three{p}))\zero{q}\ot \tuno{\one{q}})\ot_{B}(\tdue{\one{q}}\ot \tdue{\one{p}})(\zero{p}\ot \tuno{\one{p}})\\
    =&(\CR(\one{p}\ot S(\two{p}))\zero{q}\ot \tuno{\one{q}})\ot_{B}(\tdue{\one{q}}\zero{p}\ot 1)\\
    =&(\CR(\one{p}\ot S(\two{p}))\zero{q}\ot \tuno{\one{q}}\tdue{\one{q}}\zero{p})\ot_{B}1_{\C}\\
    =&(\CR(\one{p}\ot S(\two{p}))q\ot \zero{p})\ot_{B}1_{\C}\\
    =&(\CR(\one{p}\ot \one{q})\zero{q}\ot \zero{p})\ot_{B}1_{\C}\\
    =&S^{-1}(X)\ot_{B}1_{\C}.
\end{align*}
For condition (3),
\begin{align*}
    \one{S(\one{X})}&\two{X}\ot_{B}\two{S(\one{X})}\\
    =&\CR^{-1}(S(\two{p})\ot\one{p})(\zero{\tuno{\three{p}}}\ot\tuno{\one{\tuno{\three{p}}}})(\tdue{\three{p}}\ot q)\ot_{B}(\tdue{\one{\tuno{\three{p}}}}\ot\zero{p})\\
    =&\CR^{-1}(S(\two{p})\ot\one{p})(\tuno{\four{p}}\ot\tuno{S(\three{p})})(\tdue{\four{p}}\ot q)\ot_{B}(\tdue{S(\three{p})}\ot\zero{p})\\
    =&(\CR^{-1}(S(\two{p})\ot\one{p})\ot q\tuno{S(\three{p})})\ot_{B}(\tdue{S(\three{p})}\ot\zero{p})\\
    =&1_{\C}\ot_{B}(\CR^{-1}(S(\two{p})\ot \one{p})q\ot \zero{p})\\
    =&1_{\C}\ot_{B}( \CR^{-1}(\one{q}\ot\one{p})\zero{q}\ot \zero{p})\\
    =&1_{\C}\ot_{B}S(X).
\end{align*}
\end{proof}

In the case of a cleft Hopf-Galois extension, we obtain an antipode on the corresponding cocycle Hopf algebroid.

\begin{cor} If $H$ is coquasitriangular then $B^e\#_\sigma H$ has a left antipode
\[  S(b\tens b' \#_\sigma h)=u^{-1}(h\o) b'\tens  b\#_\sigma S(\two{h}),\quad S^{-1}(b\tens b' \#_\sigma h)=v(h\t) b'\tens b\#_\sigma S^{-1}(\one{h}), \]
where $v(h):=\CR(h\o\tens S(h\t))$ and $u^{-1}(h):=\CR(S^2(h\t)\tens h\o)$ are standard convolution-invertible maps related to the square of the antipode on $H$.
\end{cor}
\begin{proof} We use the isomorphism in Lemma~\ref{EScleftbialg} to transfer back $S,S^{-1}$ on $\CL(B\#_\sigma H,H)$ computed from  Proposition~\ref{prop. antipode for coquasi} as
\begin{align*}
    S(b\#_{\sigma}\one{h}\ot b'\#_{\sigma}S(\two{h}))=&b'\#_{\sigma}S(\four{h})\ot b\#_{\sigma}\one{h}\CR^{-1}(S(\three{h})\ot \two{h})\\
    =&b'\#_{\sigma}S(\four{h})\ot b\#_{\sigma}\one{h}\CR(S^{2}(\three{h})\ot \two{h})\\
    =&b'\#_{\sigma}S(\four{h})\ot b\#_{\sigma}S^{2}(\three{h})\CR(S^{2}(\two{h})\ot \one{h}),
\end{align*}
where the last step uses \cite[Lem. 2.2.2, Prop. 2.2.4]{Ma:book}. For the inverse, writing  $g=S^{-1}(h)$,
\begin{align*}
    S^{-1}(b\#_{\sigma}\one{h}\ot b'\#_{\sigma}S(\two{h}))=&b'\#_{\sigma}S(\four{h})\ot b\#_{\sigma}\one{h}\CR(\two{h}\ot S(\three{h}))\\
    =&b'\#_{\sigma}S^{2}(\one{g})\ot b\#_{\sigma}S(\four{g})\CR(S(\three{g})\ot S^{2}(\two{g}))\\
    =&b'\#_{\sigma}S^{2}(\one{g})\CR(\three{g}\ot S(\two{g}))\ot b\#_{\sigma}S(\four{g})\\
    =&b'\#_{\sigma}\three{g}\CR(\two{g}\ot S(\one{g}))\ot b\#_{\sigma}S(\four{g})\\
    =&b'\#_{\sigma}\CR(S^{-1}(\three{h})\ot \four{h})S^{-1}(\two{h})\ot b\#_{\sigma}\one{h}
\end{align*}
where the 2nd step use     of \cite[Lem.~2.2.2]{Ma:book} and the
3rd step uses \cite[Prop.~2.2.4]{Ma:book}. We use invariance of $\CR$ under $S\tens S$ to recognise the answer in terms of the Drinfeld maps  $u^{-1},v$ in \cite[Prop.~2.2.4]{Ma:book}.\end{proof}

Note that this construction is very different from the one for a cocycle Hopf algebroid of associative type in Proposition~\ref{prop. antipode of cleft extension}, hence it is clear that $S$ obeying the conditions in Lemma~\ref{lemS} need not be unique.
Indeed, for the two antipodes to coincide, we would need
\begin{align*}
    \sigma^{-1}(S^2(\two{h}), S(\one{h}))(S^2(\three{h})\triangleright b)=&u^{-1}(h)b,\quad  (\one{h}\la b)\sigma(\two{h}, S(\three{h}))=v(h) b,
\end{align*}
for all $h\in H$ and $b\in B$. Setting $b=1$, this requires  $\sigma^{-1}(S^2 (h\t),S(h\o))=u^{-1}(h)$ and $\sigma(h\o, Sh\t)=v(h)$ for all $h$, which requires $\sigma$ to be $k$-valued and $h\la b=\eps(h)b$, the trivial action. For an example, we can let $\sigma=\CR$ be viewed as a cocycle with values in $k$ which, together with the trivial action means $B\#_\sigma H=B\tens ({}_\sigma H)$, where ${}_\sigma H$ is a version of (\ref{eqn. P cotwist}) one-sided cotwisted on the left with the left coaction of $H$ on itself, i.e.
\begin{align}\label{sigmaH}
    h\cdot g=\sigma(\one{h}\ot \one{g})\two{h}\two{g}
\end{align}
and with $\sigma=\CR$ in the present case. Thus, the two antipodes can agree in very special cases, but in general are different.

\section{More examples of associative type}\label{sec5}

There are several ways to arrive at a cleft extension of associative type. The most obvious is $H$ cocommutative and $\sigma$ has its image in $Z(B)$. In this case, the condition in Definition~\ref{def: assoc type} is clearly automatic.

\subsection{Galois objects}

The next most obvious case is $\la$ trivial  $h\la b=\eps(h) b$ and in this case (\ref{eqn. assoc type}) tells us that the extension is of associative type iff the image of $\sigma$ is in $Z(B)$. We saw how this arises from gauge transform by an inner action in Example~\ref{ex inner}, but clearly one could have a more general form of $\sigma$.

The simplest such example is that of a {\em Galois object}. By definition, this as a Hopf-Galois extension where $B=k$, the field.
%\blu{We know by [R. G¡§unther, Crossed products for pointed Hopf algebras] and [H.F. Kreimer, P.M. Cook II, Galois theories and normal bases], if $H$ is a pointed Hopf algebra or finite-dimensional Hopf algebra, then any H-Galois
%object is cleft.}
 In the cleft case,  it means $P\cong H$ as a vector space and clearly in this case $\la$ is trivial and the conditions for a cocycle $\sigma$ reduce to those of a Drinfeld 2-cocycle. Then every such Galois object is a 1-sided cotwist
\[  B\#_\sigma H={}_\sigma H\]
as in (\ref{sigmaH}).

By contrast, the associated Hopf algebroid is now just a Hopf algebra and we clearly have from Proposition~\ref{prop cocy Hopf algebroid} and Theorem \ref{theorem. 2-cocycle twist} that
\[  B^e\#_\sigma H=(B^e\# H)^{\tilde{\sigma}}= H^{\sigma},\]
where the second step uses $B^{e}\# H=H$ as a Hopf algebra and $\tilde{\sigma}=\sigma$ by Lemma \ref{lem. 2-cocycle on bialgebroid} as a Drinfeld cotwist. This is in agreement with both \cite{schau},  where Galois objects were looked at, and with our twisting Theorem~\ref{theorem. 2-cocycle twist} with $\tilde\sigma=\sigma$ as the base is trivial.

This also agrees with our observations about Drinfeld cotwists. Thus, if we start with a cleft Galois object above with cocycle $\sigma$ and structure quantum group $H$, and then cotwist by $\chi$ then
\[  B\#_{\sigma*\chi^{-1}} H^\chi={}_{\sigma*\chi^{-1}} (H^\chi)={}_\sigma H_\chi,\quad B^e\#_{\sigma*\chi^{-1}} H^\chi= (H^{\chi})^{\sigma*\chi^{-1}}=H^{\sigma},\]
where $H_\chi$ is one-sided right cotwist using (\ref{eqn. P cotwist}) and the right coaction by the coproduct. We used Proposition~\ref{prop. twist cleft Galois object} for the twisted cocycle extension and then our previous observations. The condition in Proposition~\ref{prop. cocy as twist}  is automatic as $\la$ is trivial.

\subsection{Affine quantum group}

The abstract structure here \cite[Lem.~6.3.15]{Ma:book} starts with algebra maps  $B\hookrightarrow P\twoheadrightarrow H$ between Hopf algebras with $P=B\rcocross H$ as a coalgebra (so with respect to a coaction $\Delta_R(h)=h\uo\tens h\ut\in H\tens B$), the inclusion given by $b\mapsto 1\tens b$ and the surjection by $b\tens h\mapsto\eps(b)h$. Suppose that $j:H\to P$ given by $j(h)=1\tens h$ obeys $bj(h)=b\tens h=j(h)b$. Then  $j^{-1}(h)=(Sj(h\uo))h\ut$ is the convolution-inverse of $j$ and makes $P$ a cleft cocycle extension with
\[ \sigma(h, g)=j(h\o)j(g\o)(Sj(h\t\uo g\t\uo)) h\t\ut g\t\ut\]
\[ \sigma^{-1}(h, g)=j(\one{h}\one{g})(Sj(\two{g}{}^{(1)})\two{g}{}^{(2)})(Sj(\two{h}{}^{(1)})\two{h}{}^{(2)}) \]
and trivial action.  Moreover, the coalgebra is cocleft and the extension is a cocycle bicrossproduct $P=B{}_\sigma\bicross H$. The concrete example  in \cite[Ex. 6.3.14]{Ma:book} is $B=\Bbb C \Bbb Z=\Bbb C[c,c^{-1}]$ and $H=U_q(L sl_2)$  the quantum loop group version of $U_q(sl_2)$. By definition, this can be taken as generated by $k,k^{-1}, e_0,f_0,e_1, f_1$ with
\[  k e_0 k^{-1}=q^{-2}e_0,\quad ke_1k^{-1}=q^2e_1,\quad kf_0k^{-1}=q^2f_0,\quad kf_1k^{-1}=q^{-2}f_1\]
\[ q^2 e_0 f_0-f_0e_0={k^{-2}-1\over q-q^{-1}},\quad q^2 e_1f_1-f_1e_1={k^2-1\over q-q^{-1}}\]
and certain $q$-Serre relations involving powers of $e_0,f_0$. The coproduct and antipode are specified by $k$ grouplike and
\[ \Delta e_0=e_0\tens k^{-1}+1\tens e_0,\quad \Delta e_1=e_1\tens k+1\tens e_1,\quad Se_{0}=-e_{0}k,\quad Se_{1}=-e_{0}k^{-1},\]
and similarly for $f_i$. The algebra has a grading $|\ |$ by the total power of $e_0,f_0$ when expressions are ordered with all $e_i$ to the left of all $f_j$, which we view as a right coaction of $B$ by $\Delta_R h= h\tens c^{| h|}$ for any $h\in U_q(Lsl_2)$ of homogeneous degree, and we make the associated cross product coalgebra. The 2-cocycle has, for example
\[ \sigma(f_0, e_0)={1-c^2\over q-q^{-1}}\]
 and  $\Bbb C\Bbb Z{}_\sigma\bicross U_q(Lsl_2)$  is isomorphic to $\widehat{U_q(sl_2)}$. (The latter has similar relations but with some central elements $c$ appearing from the cocycle.)

Here, we observe that by the same methods, one can compute that
\[ \sigma(e_{1}, k^{\pm 1})=\sigma(k^{\pm1}, e_{0})=\sigma(e_{0},  k)=0,\quad \sigma(k^{-1},  k)=\sigma(k,  k^{-1})=1\]
\[\sigma^{-1}(k^{\pm1},  S e_{0})=\sigma^{-1}(Se_{0},  k^{\pm1})=0,\quad \sigma^{-1}(k^{-1},  k)=\sigma^{-1}(k,  k^{-1})=1.\]
Hence the antipode of the cocycle Hopf algebroid $B^e\#_\sigma H$ by  Proposition \ref{prop. antipode of cleft extension}, simplifies to
\begin{align*}
    S(b\ot b'\#_{\sigma} e_{0})=&b'\ot \sigma^{-1}(S^{2}(k^{-1}), S(e_{0}))b\#_{\sigma}S(k^{-1})+b'\ot \sigma^{-1}(S^{2}(e_{0}), S(1))b\#_{\sigma}S(e_{0})\\
    &\quad+b'\ot \sigma^{-1}(S^{2}(1), S(1))b\#_{\sigma}S(e_{0})\\
    =&b'\ot \sigma^{-1}(S^{2}(1), S(1))b\#_{\sigma}S(e_{0})=b'\ot b\#_{\sigma}S(e_{0})\end{align*}
    and similarly
    \begin{align*}
    S^{-1}(b\ot b'\#_{\sigma} e_{0})=b'\ot b\#_{\sigma}S^{-1}(e_{0})
\end{align*}
i.e. just given by flip and $S^{\pm1}$ on $H$. The same applies with $e_{0}$  replaced by $e_{1}, f_{i}$ and $k^{\pm}$. We also know from our twisting theory that this Hopf algebroid is isomorphic to a cotwist of $B^e\tens H$ (as the action is trivial) by a cotwist $\tilde\sigma$ as in Lemma \ref{lem. 2-cocycle on bialgebroid}, for example
\begin{align*}
    \tilde{\sigma}(b\ot b'\# e_{0}, d\ot d'\# f_{0})=\frac{bdd'b'(1-c^{2})}{q-q^{-1}}.
\end{align*}

%\[S^{\tilde{\sigma}}(b\ot b'\# e_{0})=b'\ot b\#S(e_{0})\quad (S^{\tilde{\sigma}})^{-1}(b\ot b'\# e_{0})=b'\ot b\#S^{-1}(e_{0}).\]

\subsection{Quantum Weyl group}

The abstract structure here (a left-right swap of \cite[Prop.~6.3.12]{Ma:book}) starts with a Hopf algebra $B$ with Drinfeld twist cocycle $\psi$  and we let ${}^\psi B$ be the twisted Hopf algebra with conjugated coproduct. We also need  $T: B\to {}^\psi B$ a Hopf algebra map such that
\[ x^{-1}ax=T^2(a),\quad T(x)=x,\quad \Delta x=(x\tens x)((T\tens T)(\psi))\psi, \quad \eps x=1
\]
for some invertible $x\in B$. Then there is a cocycle cross product algebra $B\#_{\sigma}\Bbb C\Bbb Z_2$ where if $w\in \Bbb C\Bbb Z_2$ is the generator of $\Bbb Z_2$, the cocycle and action are  \[w\la b=T(b),\quad \sigma(w, w)=x^{-1}\quad \sigma(w, 1)=\sigma(1, w)=\sigma(1, 1)=1.\]
The result is then a cocycle bicrossproduct Hopf algebra $B{}_\sigma\bicross^{\psi} \Bbb C\Bbb Z_2$ with coproduct
\[ \Delta(b\tens 1)=  b\o\tens 1\tens  b\t\tens 1,\quad \Delta(b\tens w)=b\o\psi^1\tens w\tens  b\t,
\psi^2\tens w \]
where $\psi=\psi^1\tens\psi^2$ (sum understood).   The concrete example in \cite[Ex 6.3.13]{Ma:book} is to take $B$ a ribbon Hopf algebra with ribbon element $\nu$ and $T$ an algebra automorphism and anticoalgebra map such that $T^2=\id$ and $T(\nu)=\nu$.  The ribbon element is central and hence so is the image of $\sigma$. We take $\psi=\CR$, the quasitriangular structure regarded as a cocycle. The result of the construction is an extension of $B$ such that $w^2=\nu$ in the extended algebra. The required `Lusztig automorphism' $T$ exists for the standard q-deformation enveloping algebras\cite{MaSoi} as well as for reduced quantum groups at roots of unity. The ribbon element in the latter case is given explicitly in \cite{LyuMa}.

%Specifically, for a reduced quantum group $B={\mathfrak u}_q(sl_2)$ at a primitive odd $n$th root of unity in the conventions of \cite{AziMa}, one has
%\[ T(K)=K^{-1},\quad T(E)=-q F,\quad T(F)=-q^{-1}E \]
%with ribbon element
%\[\nu={\sum_{c=0}^{n-1}q^{c^2}\over n}\sum_{a,b=0}^{n-1} {(q^{-1}-q)^a\over[a]_{q^{-1}}}q^{m(m+1)(a-b-1)^2} F^a K^b E^a\]
%if $n=2m+1$ and $[n]_q=(1-q^n)/(1-q)$ is the asymmetric $q$-integer. Here $p=q^{-m}$ is such that $p=q^2$, $q^{m(m+1)}=p^{-{1\over 2}}$ and ${\mathfrak u}_q(sl_2)\cong u_p(sl_2)$ in the more usual conventions as explained in \cite{AziMa}.}

Here, we compute the antipode for the associated cocycle Hopf algebroid $B^e\#_{\sigma}\Bbb C\Bbb Z_2$ using
 Proposition \ref{prop. antipode of cleft extension} as
\begin{align*}
    S^{\pm 1}(b\ot b'\#_{\sigma} 1)&=b'\ot b\#_{\sigma} 1\\
        S(b\ot b'\#_{\sigma} w)&=b'\ot \sigma^{-1}(w\ot w)T(b)\#_{\sigma} w^{-1}=b'\ot \nu^{-1}T(b)\#_{\sigma} w,\\
    S^{-1}(b\ot b'\#_{\sigma} w)&=T(b')\sigma(w\ot w)\ot b\#_{\sigma} w^{-1}=T(b')\nu\ot b\#_{\sigma} w.
\end{align*}
Also, we know that this Hopf algebroid is a cotwist of $B^e\#\Bbb C\Bbb Z_2$ by a cocycle in Lemma~\ref{lem. 2-cocycle on bialgebroid},
\begin{align*}
\tilde{\sigma}(b\ot b'\# 1, c\ot c'\# 1)&=bcc'b',  \quad \tilde{\sigma}(b\ot b'\# 1, c\ot c'\# w)=bc(w\la c')b'=bcT(c')b',\\
\tilde{\sigma}(b\ot b'\# w, c\ot c'\# 1)&=b(w\la (cc'b'))=bT(cc'b'),\\
    \tilde{\sigma}(b\ot b'\# w, c\ot c'\# w)&=b(w\triangleright c)\sigma(w\ot w)((w)^{2}\triangleright c')(w\triangleright b')=bT(c)\nu c'T(b').
\end{align*}
with antipode
 \begin{align*}
    (S^{\tilde{\sigma}})^{\pm}(b\ot b'\# 1)&=b'\ot b\# 1,\\
    (S^{\tilde{\sigma}})(b\ot b'\# w)=&b'\sigma^{-1}(w\ot w)\ot T(b)\# w^{-1}=b'\nu^{-1}\ot T(b)\# w,\\
    (S^{\tilde{\sigma}})^{-1}(b\ot b'\# w)=&T(b')\ot b(w\la\sigma(w\ot w))\# w
    =T(b')\ot b\nu\# w,
\end{align*}
where $w^{-1}=w$ and in the last step we use $w\la \nu=\nu$.
\subsection*{\bf Acknowledgment:} XH would like to thank Prof. Dr. Giovanni Landi, Dr. Song Cheng and A. Ghobadi for discussions. XH was supported by Marie Curie Fellowship HADG - 101027463 agreed between QMUL  and the  European Commission.


\begin{thebibliography}{99}

\bibitem{ppca}
P. Aschieri, P. Bieliavsky, C. Pagani, A. Schenkel,
{Noncommutative principal bundles through twist deformation}.
Commun. Math. Phys.  352 (2017) 287--344.

\bibitem{BegMa}
E.J. Beggs and S. Majid, {Quantum Riemannian Geometry}, Grundlehren der mathematischen Wissenschaften, Vol. 355, Springer (2020) 809pp.

\bibitem{Bich}
J. Bichon,
{Hopf-Galois objects and cogroupoids},
Rev. Un. Mat. Argentina, 5 (2014) 11--69.

\bibitem{BC}
J. Bichon and G. Carnovale
{Lazy cohomology: an analogue of the Schur multiplier for arbitrary Hopf algebras},
J. Pure and Appl. Algebra, 204 (2006) 627--665.


\bibitem{Boehm}
G. B\"ohm,
{Hopf algebroids}, in Handbook of Algebra, Vol. 6,
North-Holland, 2009, pp173--235.

\bibitem{BS}
G. B\"ohm, K. Szlach\'anyi
{Hopf algebroids with bijective antipodes:
axioms, integrals, and duals}, J. Algebra, 274 (2004) 708--750.

\bibitem{brz-tr}
T. Brzezi\'nski,
{Translation map in quantum principal bundles},
J. Geom. Phys. 20 (1996) 349--370.

\bibitem{BrzMa}
T. Brzezi\'nski and S. Majid, {Quantum group gauge theory on quantum spaces}, Commun. Math. Phys. 157 (1993) 591--638.


\bibitem{BW}
T. Brzezi\'nski, R. Wisbauer,
{Corings and comodules},
London Mathematical Society Lecture Notes Vol 309, CUP 2003.

\bibitem{D89}
Y. Doi,
{Equivalent crossed products for a Hopf algebra}, Commun. Alg. 17 (1989), 3053--3085.

\bibitem{DT86}
Y. Doi, M. Takauchi,
{Cleft comodule algebras for a bialgebra}, Commun. Alg. 14 (1986), 801--818.

\bibitem{AB}
A. Ghobadi,
{Hopf algebroids, bimodule connections and noncommutative geometry}, arXiv:2001.08673 [math.qa].

\bibitem{HL20}
X. Han, G. Landi,
{Gauge groups and bialgebroids},
Lett. Math. Phys. Vol 111, article number 140 (2021).


\bibitem{kassel-review}
C. Kassel,
{Principal fiber bundles in non-commutative geometry}, in Quantization, Geometry and Noncommutative Structures in Mathematics and Physics,
Mathematical Physics Studies, Springer 2017, pp. 75--133.

\bibitem{LyuMa}V. Lyubashenko and S. Majid, {Braided groups and quantum Fourier transform}, J. Algebra. 166 (1994) 506--528

\bibitem{KirillMackenzie}
K. Mackenzie,
{General theory of Lie groupoids and Lie algebroids}.
London Mathematical Society Lecture Notes Vol 213, CUP 2005.



\bibitem{Ma:mor}S. Majid, {More examples of bicrossproduct and double cross product Hopf algebras}, Isr. J. Math. 72 (1990) 133--148.

\bibitem{Ma:book}
S. Majid,
{Foundations of quantum group theory},
 CUP 1995 and 2000.

\bibitem{Ma:non}
S. Majid,
{Cross product quantisation, nonabelian cohomology and twisting of Hopf algebras}, in  Generalized Symmetries in Physics, World Sci. (1994) 13-41 (arXiv:hep-th/9311184v1)

\bibitem{Masuoka}
A. Masuoka,
{Cleft extensions for a Hopf algebra generated by a nearly primitive element},
Comm. Alg. 22 (1994) 4537--4559.



\bibitem{MaSoi}S. Majid and Ya. S. Soibelman, {Bicrossproduct structure of the quantum Weyl group}, J. Algebra 163 (1994) 68--87.

\bibitem{mont}
S. Montgomery,
{Hopf algebras  and their actions on rings},  AMS, 1993.

\bibitem{schau}
P. Schauenburg,
{Hopf bi-Galois extensions},
Commun. Alg. 24 (1996) 3797--3825.

\bibitem{schau1}
P. Schauenburg,
{Duals and doubles of quantum groupoids ($\times_R$ -Hopf algebras)}, in New
Trends in Hopf Algebra Theory, AMS Contemp.
Math. 267 (2000) 273--299.

\bibitem{Swe}M. Sweedler, {Hopf Algebras}, Benjamin, 1969.

\bibitem{kornel}
K. Szlachanyi,
{Monoidal Morita equivalence},  AMS Contemp. Math. 391 (2005) 353--369.


\end{thebibliography}
\end{document}